\documentclass[12pt]{amsart}
\usepackage{epsfig,color}
\usepackage{amsmath, amssymb}

\usepackage{mathrsfs}

\theoremstyle{definition}

\newtheorem{theorem}{Theorem}[section]

\newtheorem{proposition}[theorem]{Proposition}
\newtheorem{lemma}[theorem]{Lemma}

\newtheorem{corollary}[theorem]{Corollary}

\newtheorem*{acknowledgements}{Acknowledgements}

\numberwithin{equation}{section}

\usepackage{enumerate}

\renewcommand\div{\operatorname{div}}
\DeclareMathOperator*{\im}{im}
\DeclareMathOperator*{\Diff}{Diff}

\newcommand{\tr}{\operatorname{tr}}

\newcommand{\wt}{\widetilde}

\newcommand{\Lap}{\Delta}

\newcommand{\Ric}{\operatorname{Ric}}

\newcommand{\Hess}{\operatorname{Hess}}

\renewcommand*\d{\mathop{}\!\mathrm{d}}

\newcommand{\cK}{\mathcal{K}}
\newcommand{\cS}{\mathcal{S}}
\newcommand{\eps}{\varepsilon}

\newcommand{\spec}{\text{spec}}

\newcommand{\vertiii}[1]{{\left\vert\kern-0.25ex\left\vert\kern-0.25ex\left\vert #1 
    \right\vert\kern-0.25ex\right\vert\kern-0.25ex\right\vert}}

\usepackage{comment}

\title{Rigidity of spherical product Ricci solitons}

\author{Ao Sun}
\address{Department of Mathematics,
University of Chicago,
5734 S. University Avenue,
Chicago, IL 60637, USA}
\email{aosun@uchicago.edu}

\author{Jonathan J. Zhu}
\address{Mathematical Sciences Institute, Australian National University, Hanna Neumann Building, Science Road, Canberra, ACT 2601, Australia}
\email{jjzhu@math.princeton.edu}

\begin{document}
\begin{abstract}
We show that $S^2\times S^2$ is isolated as a shrinking Ricci soliton in the space of metrics, up to scaling and diffeomorphism. We also prove the same rigidity for $S^2\times N$, where $N$ belongs to a certain class of closed Einstein manifolds. These results are the Ricci flow analogues of our results for Clifford-type shrinking solitons for the mean curvature flow. 
\end{abstract}
\date{\today}
\maketitle

\section{Introduction}

Ricci solitons are metrics that flow by rescaling (up to diffeomorphism) under the Ricci flow. They were first introduced by Hamilton \cite{H1} to study the singular behaviour of the Ricci flow. \textit{Shrinking} solitons are those that contract under Ricci flow, and relate to the \textit{formation} of singularities. A gradient shrinking Ricci soliton is a Riemannian manifold $(M,g)$ that satisfies the equation
\begin{equation}\label{eq:grad-soliton} \Ric(g) +\Hess^g f=\frac{1}{2\tau }g \end{equation} 
for some smooth function $f$ on $M$ and some $\tau>0$. Note that equation (\ref{eq:grad-soliton}) has two natural symmetries: diffeomorphisms of $M$, and rescaling (which results in a rescaled $\tau$). In the compact setting, every closed shrinking Ricci soliton is a gradient shrinking soliton. Shrinking solitons often arise as Type I singularity
models of Ricci flow, and also as tangent flows, which were recently introduced by Bamler in \cite{B}.

In this paper, we study the rigidity of certain gradient shrinking Ricci solitons $(M,g)$, that is, whether $(M,g)$ is isolated in the space of shrinking Ricci solitons (modulo diffeomorphism and rescaling). Our first main theorem is the following local rigidity theorem for $S^2\times S^2$.

\begin{theorem}\label{thm:main1}
	Let $(M,g)$ be $S^2\times S^2$ with the round unit metric on each factor. There is a $C^{2,\alpha}$-neighbourhood $\mathcal{U}$ of $g$ such that any other gradient shrinking Ricci soliton in $\mathcal{U}$, is equivalent to $g$ up to diffeomorphism and rescaling. 
\end{theorem}

The question of local rigidity of $S^2\times S^2$ was previously posed by Kr\"{o}ncke \cite[Example 6.4]{Kr2}. Our approach here is the Ricci flow analogue of our methods in \cite{SZ} for $S^m\times S^n$ as mean curvature flow solitons (see also \cite{ELS},  \cite{Z20}). These methods are needed to handle non-integrable infinitesimal deformations. In contrast to the MCF setting, as Ricci solitons such deformations only arise when at least one spherical factor has dimension $2$.

Our methods also can be used to show local rigidity for a class of product manifolds $S^2\times N$, where $N$ is a 1-Einstein manifold satisfying 
	\begin{itemize}
	\item[($\dagger$)] $\lambda_1(\Lap^N)>2$, and $\ker_{TT}(\Lap_L^N+2)=0$. 
	\end{itemize}
	Here we use the sign convention $\Lap=\div\nabla$ for the Laplacian on functions, $\lambda_1$ is the first nonzero eigenvalue and $\Lap_L$ is the Lichnerowicz Laplacian; $\ker_{TT}$ refers to the kernel when acting on transverse traceless 2-tensors. 
	We prove the following: 

\begin{theorem}\label{thm:main2}
	Let $(M,g)$ be $S^2\times N$, where $S^2$ has the round unit metric and $(N,g_N)$ is a closed $1$-Einstein manifold satisfying $(\dagger)$. There is a $C^{2,\alpha}$-neighbourhood $\mathcal{U}$ of $g$ such that any other gradient shrinking Ricci soliton in $\mathcal{U}$, is equivalent to $g$ up to diffeomorphism and rescaling. 
\end{theorem}

There are many $N$ which satisfy the assumptions in Theorem \ref{thm:main2}. For example, any $N$ listed in \cite[Table 2]{CH}, that is not listed as ``i.d." and has (in their notation) $-\lambda^{-1}\mu_{fns} >2$. 

Both theorems above are direct consequences of the following quantitative rigidity theorem. For this we introduce the shrinker quantity 
\begin{equation}
\Phi(g):= \tau^g(\Ric(g) + \Hess^g f_g) - \frac{g}{2},
\end{equation}
where $f^g,\tau^g$ are the pair realising the Perelman shrinker entropy $\nu$; an important property is that $\Phi$ is equivariant under pullback by diffeomorphisms, and under rescalings (see Section \ref{sec:prelim}). Any compact gradient shrinking Ricci soliton must satisfy (\ref{eq:grad-soliton}) with $(f,\tau) = (f_g,\tau_g)$, that is, $\Phi(g)=0$ (see \cite[Remark 3.4]{Kr3}). 

\begin{theorem}
\label{thm:main3}

		Let $(M,g)$ be either $S^2\times S^2$ or $S^2\times N$, where $S^2$ has the round unit metric and $(N,g_N)$ is a closed $1$-Einstein manifold satisfying $(\dagger)$. There exist $C,\epsilon_0>0$ so that if $\tilde{g}$ is a metric on $M$ with $\|\tilde{g}-g\|_{C^{2,\alpha}} \leq \eps_0$, then there exists $c \in (1-\frac{1}{C}, 1+\frac{1}{C})$ and a diffeomorphism $\psi$ of $M$ so that, 
		\begin{equation}
			\|c \psi^*\tilde{g}-g\|_{C^{2,\alpha}}^{3}\leq C\|\Phi(\tilde{g})\|_{C^{0,\alpha}},
		\end{equation}
	
\end{theorem}

As $\Phi$ may be regarded as the $L^2$-gradient of the shrinker entropy $\nu$, Theorem \ref{thm:main3} may be regarded as the H\"{o}lder version of a {\L}ojasiewicz inequality for $\nu$ about the critical points $S^2\times S^2$, $S^2\times N$. The $L^2$ version is somewhat more natural and would yield a gradient {\L}ojasiewicz inequality as well, but the theory is significantly more technical so we leave it for future work; the H\"{o}lder version is more than sufficient for rigidity. 

%
%

\subsection*{Background}

The question of rigidity - whether a geometric object is isolated in its class, modulo the symmetries of that class - has been studied for various geometric structures in many contexts. For Ricci solitons, low dimensional solitons have been essentially classified, so their rigidity or non-rigidity is clear (the dimension $2$ case is due to Hamilton \cite{H2} and the dimension $3$ case is due to Ivey \cite{I}). In higher dimensions, the classification is far from well-understood. From a geometric PDE perspective, Ricci solitons are natural generalizations of Einstein manifolds. Kr\"oncke studied the local rigidity of Ricci solitons in \cite{Kr2}, using similar techniques as for Einstein manifolds in \cite{Kr1}. 

For non-rigidity results, similar to the Einstein case, Inoue \cite{In} studied the moduli space of Fano manifolds with Kähler-Ricci solitons. As a consequence, Inoue proved that some Kähler-Ricci solitons are non-rigid. It seems that the local rigidity of a Ricci soliton actually has deep connections to the topological/algebraic structure of the manifold, and we look forward to such interpretations for the class of Ricci solitons we study in this paper.

The study of rigidity of Einstein manifolds has a more developed history. The first rigidity result for Einstein manifolds was proved by Berger in \cite{Be1}, where Berger proved that all Einstein metrics on $S^n$ whose sectional curvature is $(n-2)(n-1)$-pinched are isometric to the standard sphere metric. We consider this to be a `global' result in that it gives a large explicit neighbourhood in which the special structure is unique. Later, Berger-Ebin \cite{BE} proposed a general strategy to prove local rigidity of Einstein metrics. In particular, they introduced the notion of ``infinitesimal deformations". The local rigidity of symmetric spaces as Einstein manifolds was first studied by \cite{Ko}, and Kr\"oncke extended Koiso's result in \cite{Kr1}. 

Meanwhile, it is known that certain moduli spaces of Einstein manifolds, even with an extra structure (e.g. K\"ahler), can be nontrivial. Then any metric in such a moduli space is an example of non-rigid Einstein manifold. For example, even in dimension 2, the moduli space of surfaces with genus $g$ of  constant curvature $-1$ has dimension $6g-6$. Hence when $g>1$, such hyperbolic metrics are not rigid. We refer the readers to \cite[Chapter 12.J]{Besse} for further discussion.

The Ricci soliton equation is both more general and more complicated than the Einstein equation, making it more challenging to obtain a rigidity result. It is also interesting to compare Ricci flow with other geometric flows, especially the most natural extrinsic flow - the mean curvature flow. There are a number of results on the rigidity of shrinking solitons to the mean curvature flow - see \cite{CIM}, \cite{ELS}, \cite{SZ}, \cite{Z20}. All these results are about product self-shrinkers. During the completion of this work, we learnt that Colding-Minicozzi are developing techniques which could be applied to study certain noncompact Ricci solitons \cite{CMrf1, CMrf2}. 

\subsection*{Deformation theory for solitons}

The starting point in the study of local rigidity is to consider infinitesimal deformations which satisfy the equation $\Phi=0$ to first order. That is, one looks for $2$-tensors in the kernel of the linearized operator $\mathcal D\Phi$. By analogy with the mean curvature flow setting, we call elements of this kernel \textit{Jacobi deformations}; in the Ricci flow literature these are sometimes referred to as infinitesimal solitonic deformations \cite{Kr2}. A Jacobi deformation is called integrable if it induces a continuous family of gradient shrinking Ricci solitons. Any diffeomorphism or rescaling generates an integrable Jacobi deformation. 

 In \cite{Kr2}, Kr\"{o}ncke found conditions under which certain Einstein manifolds are rigid as gradient shrinking Ricci solitons. His idea was to expand $\Phi =0$ to higher and higher order and examine whether there are solutions to each order. If there are no solutions at a given order, we call this an obstruction of that order. Kr\"{o}ncke calculates only up to order 2: $\mathbb{CP}^{2n}$ has a second order obstruction, whilst for $S^2\times S^2$ all Jacobi deformations are integrable to second order. (For $\mathbb{CP}^{2n+1}$, almost all Jacobi deformations are non-integrable to second order, but there is a set of positive codimension which is indeed integrable.)
 
 The manifolds we study in this paper, also have Jacobi deformations which are integrable to order 2, but we will show they have an obstruction at order 3. In particular, we resolve the case of $S^2\times S^2$. We also hope that our method may be adapted to determine the rigidity of $\mathbb{CP}^{2n+1}$. As described below, expanding at orders higher than 2 result in \textit{inhomogenous} elliptic PDEs, which makes the analysis significantly more complicated. For instance, we use the special structure of eigenfunctions on $S^2$ to solve (\ref{eq:second-order-intro}) explicitly. 
   
 In \cite{SZ}, we studied the rigidity of products of spheres $S^m\times S^n$ as mean curvature flow self-shrinkers, and found that despite the existence of nontrivial Jacobi deformations, there is an obstruction at third order. (Note that \cite{ELS}, by a similar method, had previously covered the $m,n=1$ case.) As mentioned above, when $m,n\geq 3$, as a Ricci shrinker $S^m\times S^n$ does not have Jacobi deformations other than those from symmetry (a first order obstruction). However, when at least one of the factors is $S^2$, nontrivial Jacobi deformations appear once more. We remark that when considered in the class of Einstein manifolds, there are again no nontrivial deformations, even for $S^2$ products.


\subsection*{Proof strategy}

Our proof strategy is a deformation-obstruction theory for the shrinker quantity $\Phi$, mirroring that we used in \cite{SZ} (also developed by the second named author in \cite{Z20}). One may consult those papers for an overview of the method, including in toy cases, but we reproduce a sketch below for the readers' convenience. The basic idea is similar to that of Kr\"{o}ncke \cite{Kr2}, but we study the expansion of $\Phi$ more directly to find a `uniformly' obstructed term, which gives the quantitative rigidity. 

Consider an elliptic functional $\mathcal{F}$ defined on a Banach space $\mathcal{E}$. Let $\mathcal{G}:\mathcal{E}\to \mathcal{E}'$ be the Euler-Lagrange operator of $\mathcal{F}$ and suppose $0$ is a critical point of $\mathcal{F}$, so that $\mathcal{G}(0)=0$. In the Ricci soliton setting, $\mathcal{F}$ will be the shrinker entropy $\nu$ and $\mathcal{G}$ will be the Ricci shrinker quantity $\Phi(g+\cdot)$. 

For $h\in \mathcal{E}$, we have the formal Taylor expansion 
\begin{equation}
\mathcal{G}(h) = \mathcal{D}\mathcal{G}(h) + \frac{1}{2}\mathcal{D}^2\mathcal{G}(h,h) + \frac{1}{6} \mathcal{D}^3 \mathcal{G}(h,h,h) + \cdots.
\end{equation}
Here $\mathcal{D}\mathcal{G}(h) = Lh$, where $L$ is the linearised operator at 0. Its kernel $\mathcal{K}$ is the space of Jacobi deformations. 

If $\mathcal{K}=0$, ellipticity gives that $L$ is invertible, hence $\mathcal{G}$ is obstructed at order 1. (As mentioned above, modulo symmetries, this is the case for $S^m\times S^n$, $m,n\geq 3$.) Otherwise $\mathcal{K}\neq 0$, and the first order term suggests the decomposition $h = h_1 + k_1$, where $k_1 \in \mathcal{K}$ and $h_1$ is in a complement $\mathcal{E}_1$ of $\mathcal{K}$. This gives the further expansion 

\begin{equation}
\mathcal{G}(h) = Lh_1 + \frac{1}{2}\mathcal{D}^2\mathcal{G}(k_1,k_1) + \mathcal{D}^2\mathcal{G}(k_1,h_1) + \frac{1}{6} \mathcal{D}^3 \mathcal{G}(h_1,h_1,h_1) + \cdots.
\end{equation}

Invertibility of $L$ on $\mathcal{E}_1$ implies that $h_1$ is (at least) second order compared to $h$ and $k_1$. The second order term is then $Lh_1 + \frac{1}{2}\mathcal{D}^2\mathcal{G}(k_1,k_1)$. By elliptic theory, there exists $h_1$ for which this term is 0, if and only if $\pi_\mathcal{K}(\mathcal{D}^2\mathcal{G}(k_1,k_1))=0$. In \cite{Kr2}, Kr\"{o}ncke essentially finds Einstein manifolds for which $\pi_\mathcal{K}(\mathcal{D}^2\mathcal{G}(k_1,k_1)) \neq 0$ for all nonzero $k_1 \in \mathcal{K}$. This would imply $\| \pi_\mathcal{K}(\mathcal{D}^2\mathcal{G}(k_1,k_1))\| \geq \delta \|k_1\|^2$, which we refer to as an order 2 obstruction, and in turn implies $\|h\| \simeq \|k_1\| \lesssim \|\mathcal{G}\|^\frac{1}{2}$. 

For the cases considered in this paper, it turns out that $\pi_\mathcal{K} (\mathcal{D}^2\mathcal{G}(k_1,k_1)) \equiv 0$. That is, there is always a solution of \begin{equation}\label{eq:second-order-intro}Lk_2 + \frac{1}{2}\mathcal{D}^2\mathcal{G}(k_1,k_1)=0\end{equation} for some $k_2 \in\mathcal{E}_1$, which suggests the further decomposition $h_1 = k_2 + h_2$. The expansion then becomes 
\begin{equation}
\mathcal{G}(h) = Lh_2 + \mathcal{D}^2\mathcal{G}(k_1,k_2) + \frac{1}{6} \mathcal{D}^3 \mathcal{G}(k_1,k_1,k_1) + \cdots,
\end{equation}
where again invertibility of $L$ shows that the residual $h_2$ is of third order compared to $h$. Repeating the process, we may attempt to solve the third order condition \begin{equation}Lh_2 + \mathcal{D}^2\mathcal{G}(k_1,h_2) + \frac{1}{6} \mathcal{D}^3 \mathcal{G}(k_1,k_1,k_1) =0,\end{equation} which depends on the projection $\pi_\mathcal{K} \left(\mathcal{D}^2\mathcal{G}(k_1,h_2) + \frac{1}{6} \mathcal{D}^3 \mathcal{G}(k_1,k_1,k_1)\right).$ Note that the third order condition involves some lower order cross terms. 

If the projection vanishes identically, we proceed to the next higher order. If at the $m$th order, the corresponding projection is nonzero for any nonzero $k_1 \in \mathcal{K}$, then we have an $m$th order obstruction, which would imply the quantitative rigidity $\|h\| \lesssim \|\mathcal{G}\|^\frac{1}{m}$. Again, for the cases considered here, we find a third order obstruction 
\begin{equation}  \left\| \pi_\mathcal{K} \left(\mathcal{D}^2\mathcal{G}(k_1,k_2) + \frac{1}{6} \mathcal{D}^3 \mathcal{G}(k_1,k_1,k_1)\right) \right\| \geq \delta \|k_1\|^3.\end{equation}

\subsection*{Organisation of the paper}

We collect some preliminary results in Section \ref{sec:prelim} and record descriptions of the Jacobi deformations on our $S^2$ products in Section \ref{sec:jacobi-einstein}. In Section \ref{sec:shrinker-variation} we compute the variation of the shrinker quantities $\nu, \tau, f, \Phi$, with the variation formulae for more basic geometric quantities recorded in Appendix \ref{sec:basic-var}. Using the explicit description of Jacobi deformations, we specialise these variations to our $S^2$ products in Section \ref{sec:S2-special} and establish a formal obstruction at order 3. Finally, in Section \ref{sec:holder}, we use Taylor expansion and the formal obstruction to prove the main quantitative rigidity Theorem \ref{thm:main3}. The a priori bounds on $\mathcal{D}^k\Phi$ are deferred to Appendix \ref{sec:shrinker-quantities-bound}.

\begin{acknowledgements}
	The authors want to thank Professors Bill Minicozzi and Toby Colding for sharing in enlightening discussions about their work and ours. JZ would also like to thank Prof. Richard Bamler for helpful conversations. We thank Prof. Klaus Kr\"{o}ncke for rectifying our understanding of his work as well as Jiangtao Yu for pointing out an inaccuracy in the prior computation. 
	
	JZ was supported in part by the Australian Research Council under grant FL150100126.
\end{acknowledgements}


\section{Preliminaries}
\label{sec:prelim}
Throughout, we consider a closed Riemannian manifold $(M,g)$. 

Our sign convention for the (rough) Laplacian is $\Lap = \div \nabla$. Here, and henceforth, we suppress dependence on the metric $g$ when clear from context. When we wish to emphasise the domain, we will use $\Lap_0$ for the Laplacian on functions, $\Lap_1$ for the rough Laplacian on 1-forms and $\Lap_2$ for the rough Laplacian on symmetric 2-tensors. 
When not otherwise specified or clear from context, $\Lap$ should be understood as the Laplacian on functions. With our sign convention, we say $u\neq0$ is an eigenfunction of $\Lap$ with eigenvalue $\lambda$ if $\Lap u =-\lambda u$, and similarly for other operators. We write $\lambda_1$ for the first nonzero eigenvalue. 

A metric is said to be $(\mu$-) Einstein if $\Ric(g) = \mu g$ for some $\mu\in\mathbb{R}$. Let $\mathcal{S}^2(M)$ denote the space of (smooth) symmetric $2$-tensors on $M$, and $\Omega^1(M)$ the space of (smooth) 1-forms on $M$. Throughout the paper, we will use several different norms like the Sobolev norm and the H\"older norm. We can also study the tensors in the completion under these norms, but approximation by smooth sections easily yields the same estimates. 

An important subspace is the space of transverse traceless (TT) $2$-tensors, which we denote by $\mathring{\mathcal{S}}^2_g = \{ h\in\mathcal{S}^2(M) | \tr_g h =0, \div_g h=0\}$. Then we have the well-known ($L^2_{g}$-) orthogonal decomposition \begin{equation}\mathcal{S}^2(M) = (C^\infty(M) g + \im\mathcal{L}) \oplus \mathring{\mathcal{S}}^2_g.\end{equation} Here $\mathcal{L}: \Omega^1(M) \to \mathcal{S}^2(M)$ is the Lie derivative, $(\mathcal{L} \omega)_{ij} = \nabla_i \omega_j + \nabla_j \omega_i$. Note that $\mathcal{L}$ is actually independent of the metric, and $-\frac{1}{2}\mathcal{L}$ is the $L^2_g$-adjoint of $\div_g$ for any $g$. 

We define the Lichnerowicz Laplacian $\Lap_L : \mathring{\mathcal{S}}^2_g \to \mathring{\mathcal{S}}^2_g$ as acting by 
\[\Lap_Lh_{ij} = \Lap h_{ij} + 2R^l_{kij} h^k_l - R^k_i h_{jk} - R^k_j h_{ik}.\] Note that this definition also makes sense on all of $\mathcal{S}^2(M)$ but we choose to restrict the domain so that $\ker(\Lap_L)=\ker_{TT}(\Lap_T)$ and so forth will refer to the TT kernel. 

Define \[\mathcal W(g,f,\tau):=\int [\tau (|\nabla f|^2+R)+f-n](4\pi\tau)^{-n/2}e^{-f}\d V_g.\]
The shrinker entropy is defined by \begin{equation}\nu(g)=\inf \left\{ \mathcal{W}(g,f,\tau) | \tau>0, (4\pi\tau)^{-n/2}\int e^{-f}\d V_g = 1\right\}.\end{equation}
As in \cite[Lemma 4.1]{Kr3}, given any gradient shrinking Ricci soliton, there is a $C^{2,\alpha}$ neighbourhood in the space of metrics on which the minimisers $f^g, \tau^g$ are unique and depend analytically on the metric; consequently $\nu$ also depends analytically on the metric on this neighbourhood. Moreover, the first variation of $\nu$ is given by 
\begin{equation}
\mathcal{D}\nu_g(h)= - (4\pi\tau_g)^{-n/2} \int \langle \Phi(g), h\rangle e^{-f_g} \d V_g,
\end{equation}
where the Ricci shrinker quantity $\Phi:\mathcal{S}^2(M)\to \mathcal{S}^2(M)$ is defined by 
\begin{equation}\Phi(g) = \tau^g(\Ric(g) +\Hess^g f^g) - \frac{g}{2}.\end{equation} 
(Compact) gradient shrinking Ricci solitons are those metrics which satisfy $\Phi(g)=0$. 
For convenience we denote $w^g=\frac{e^{-{f^g}}}{(4\pi {\tau^g})^{n/2}}$. 
By definition, we have the normalisation $\int w^g \d V_g=1$. Henceforth, for readability, we will suppress dependence on the metric where clear from context. 

Any Einstein manifold $(M,g)$ is a gradient shrinking Ricci soliton with $f_g$ constant. Henceforth, $L^2$ will always refer to the Lebesgue space weighted by $w^g$; if $g$ is Einstein this is equivalent to $L^2_g$ up to the constant weight $w$. 

It is known that (see \cite{CCG+}) the Euler-Lagrange equations for $f,\tau$ are  \begin{equation}\label{eq:EL-nu} \tau(-2\Lap f + |df|^2 -R) - f+n +\nu=0.\end{equation}
\begin{equation}
\label{eq:EL-nu2}
\int fw\d V_g=\frac{n}{2}+\nu.
\end{equation}

Through this paper, $C,C',C''$ will denote constants which may change from line to line but retain their stated dependencies. 

We record the following observation: 

\begin{lemma}
\begin{equation}\label{eq:divPhi}\div_g \Phi(g) = \Phi(g)\cdot \nabla^g f.\end{equation}
\end{lemma}
\begin{proof}
Using the Bianchi identity gives $\div \Ric = \frac{1}{2}\nabla R$. By commuting derivatives we have $\div \Hess f = \nabla \Lap f + \Ric\cdot \nabla f = \nabla \Lap f + \frac{1}{\tau}\Phi \cdot \nabla f + \frac{1}{2\tau}\nabla f - \Hess f \cdot \nabla f$. On the other hand, differentiating the defining equation for $f$ in space gives \[\tau(-2\nabla \Lap f + 2\Hess f \cdot \nabla f - \nabla R) - \nabla f=0.\] This implies the result. 
\end{proof}

\subsection{Variations of shrinker quantities}

Consider a soliton metric $\Phi(g)=0$. We consider the formal expansion of $\Phi$ given by 

\[ \Phi(g+h) =  \mathcal{D}\Phi_g(h) + \frac{1}{2} \mathcal{D}^2\Phi_g(h,h) + \frac{1}{6}\mathcal{D}^3\Phi_g(h,h,h)+\cdots \]

where $\mathcal{D}^k\Phi_g$ is a symmetric $k$-linear map $\mathcal{S}^2(M)^{k} \to \mathcal{S}^2(M)$. For instance, $L=\mathcal{D}\Phi_g$ is the linearisation of $\Phi$ at $g$. We will henceforth suppress dependence on the initial metric $g$. We will use the corresponding notation for variations of other quantities such as $\nu(g), \tau_g, f_g$. 

\subsection{Deformations and Jacobi deformations}

Given a 1-parameter family of metrics $g(s)$ with $g(0)=g$, the corresponding infinitesimal deformation is $g'(0) \in \mathcal{S}^2(M)$. At a shrinking soliton $\Phi(g)=0$, we define $\mathcal{K}:=\ker L$ to be the kernel of $L=\mathcal{D}\Phi_g$. We call elements of $\mathcal{K}$ Jacobi deformations; they correspond to deformations which preserve $\Phi=0$ to first order and have previously been referred to as infinitesimal solitonic deformations.  

The Ricci shrinker quantity has two natural symmetries: under scaling by $c>0$ we have $\Phi(cg)=c\Phi(g)$, and under pullback by a diffeomorphism $\psi\in\Diff(M)$ we have $\Phi(\psi^* g) = \psi^*\Phi(g)$. The corresponding spaces of infinitesimal deformations are $\mathbb{R} g$ and $\im\mathcal{L}$, and if $\Phi(g)=0$ these are integrable subspaces of Jacobi deformations. We define the subspace of Jacobi deformations from symmetry to be $\mathcal{K}_0 :=\im\mathcal{L} + \mathbb{R}g \subset \mathcal{K}$. 

If $g$ is a $\mu$-Einstein metric, Kr\"{o}ncke \cite[Section 6]{Kr2} checked that \[\ker\div_g = \{(\Lap v + \mu v)g + \Hess^g v\}.\] Note that $\Hess^g v = \frac{1}{2}\mathcal{L}(\nabla v)^\sharp$. A theorem of Lichnerowicz (\cite{Li}, see also \cite[Theorem 2.1]{IV}) implies that $\lambda_1(\Lap) >\mu$, and in particular  that $\Lap+\mu$ is invertible. 

Let $\mathring{C}^\infty_g = \{u \in C^\infty(M) | \int v \d V_g=0\}$, so that the conformal deformation space satisfies \begin{equation}C^\infty(M) g + \im \mathcal{L} = (\Lap+\mu)\mathring{C}^\infty_g\cdot g + \mathcal{K}_0.\end{equation} Without loss of generality, we will assume $\mu=1$. The action of $L$ is well-known and also computed in Section \ref{sec:1st-variation}: It acts on the conformal part $vg$, where $v=(\Lap+1)\tilde{v}$, $\tilde{v}\in \mathring{C}^\infty_g$, by \[ L(vg) = -\frac{1}{4}(\Lap+2)vg + \frac{1}{4} \Hess(\Lap+2)\tilde{v},\] and on the TT part $h\in \mathring{\mathcal{S}}^2_g$ by \[L(h) = -\frac{1}{4}(\Lap_L+2)h.\]

Recall that we defined $\Lap_L$ as an operator on $\mathring{\mathcal{S}}^2_g$. The above implies the $L^2_g$-orthogonal decomposition of the kernel \begin{equation}\mathcal{K} = (\mathcal{K}_0 + \mathcal{K}_1 )\oplus \ker (\Lap_L+2),\end{equation} 
where \begin{equation}\mathcal{K}_1:= \{ vg | v\in\ker(\Lap+2)\}.\end{equation}

We denote by $\pi_\mathcal{K}$ the $L^2$-projection to $\mathcal{K}$, and similarly for other subspaces. 

To study deformations of a gradient shrinking Ricci soliton, we need to account for the symmetry action at an integrated level, not just infinitesimal. More precisely, given a gradient shrinking Ricci soliton $(M,g)$ and a metric $\tilde{g}$ close to $g$, we want to compare $\tilde{g}$ to the symmetry orbit of $g$ (or vice versa). 

\begin{lemma}
\label{lem:slice}
	Suppose $(M,g)$ is a $\mu$-Einstein manifold. For any $\delta>0$ there exists $\epsilon_0$ such that if $\|\tilde{g}-g\|_{C^{2,\alpha}}\leq \eps_0$, then there exists $c \in (1-\delta, 1+\delta)$ and $\psi \in \Diff(M)$ such that \[h:= c\psi^* \tilde{g}-g \in (\Lap+\mu)\mathring{C}^\infty_g\cdot g \oplus \mathring{\mathcal{S}}^2_g,\] and moreover $\|h\|_{C^{2,\alpha}}\leq \delta$.
\end{lemma}

The proof is standard, using implicit function theorem for Banach spaces. We refer the readers to the proof of Theorem 3.6 in \cite{V}.

\section{Jacobi deformations on Einstein manifolds}
\label{sec:jacobi-einstein}

In this section we discuss Jacobi deformations on products of Einstein manifolds, with the goal of giving an explicit description of the Jacobi deformations on $S^2\times S^2$ and certain products $S^2\times N$. We also demonstrate that there are no Jacobi deformations on $S^m\times S^n$, $m,n\geq 3$, other than those from scaling and diffeomorphism. 

\subsection{Spectrum on product Einstein manifolds}
By a standard separation of variables, spectra on a product manifold can be derived from the spectra of each component. For the Laplacian on functions, we have
\begin{equation} \spec(\Lap^{M\times N}_0) = \spec(\Lap^M_0) + \spec(\Lap^N_0).\end{equation}
For the spectrum of $\Lap_L$, it is convenient to first work on all of $\mathcal{S}^2(M)$. To match with previous literature, we define the Einstein operator $\Lap_E: \mathcal{S}^2(M) \to \mathcal{S}^2(M)$ by 
\[\Lap_E h_{ij} = \Lap h_{ij} + 2R^l_{kij} h^k_l.\]
Note that this differs from the expression for $\Lap_L$ by exactly the Ricci terms; in the case of a $\mu$-Einstein manifold, for $h\in \mathring{\mathcal{S}}^2_g$ we have $\Lap_Eh = (\Lap_L+2)h$. 

We have the following theorem (see \cite{AM} and \cite{Kr4}; note that our sign convention is opposite to Kr\"{o}ncke's).
\begin{theorem}
\label{thm:spec-prod}
Suppose $M,N$ are $\mu$-Einstein manifolds. Then
	\begin{equation}\spec(\Delta_E^{M\times N})=(\spec(\Delta_E^M)+\spec(\Delta_0^N))\cup (\spec(\Delta_E^N)+\spec(\Delta_0^M))\cup(\spec(\Delta_1^M)+\spec(\Delta_1^N)).\end{equation} 
\end{theorem}
Moreover, the eigentensors on the product manifold may be given by (symmetric tensor) products of the corresponding eigensections. 
 

\subsection{Spectra on round spheres} 
\label{sec:sph-spec}
We will always consider $S^n$ with the metric induced as the sphere of radius $\sqrt{n-1}$ in $\mathbb{R}^{n+1}$, which is $1$-Einstein. The spectra of the Laplace operators has been computed (see for instance \cite{Bou99, Bo}):

 For $k,n\in\mathbb{N}$ define $\lambda^0_{k,n} = \frac{k(k+n-1)}{n-1}$, $\lambda^1_{k,n} = \frac{k(k+n-1) + n-2}{n-1}$ and $\lambda^2_{k,n} = \frac{k(k+n-1) +2(n-1)}{n-1}$.

\begin{itemize}
	\item The spectrum of the Laplacian on functions $\Lap_0$ consists of eigenfunctions $\varphi_{k,n}$ with eigenvalue $\lambda^0_{k,n}$, for $k\geq 0$. 
	\item The Hodge Laplacian acts on $\Omega^1(S^n)$ by $\Lap_1 -1$ and has spectrum consisting of $d\varphi_{k,n}$, $k\geq 1$, together with divergence-free eigenforms $\omega_{k,n}$ with eigenvalue $\lambda^1_{k,n}$, for $k\geq 1$.
	\item The (generally defined) Lichnerowicz Laplacian acts on $\mathcal{S}^2(S^n)$ by $\Lap_E - 2$, and its spectrum consists of: $\varphi_{k,n} g$, $k\geq 0$; $\Hess \varphi_{k,n}$, $k\geq 2$; $\mathcal{L}\omega_{k,n}$, for $k\geq 2$; and TT eigenforms $h_{k,n}$ with eigenvalue $\lambda^2_{k,n}$, for $k\geq 2$. 
\end{itemize}

%
%

That is, 
	\[\spec(\Lap^{S^n}_0) = \{\lambda^0_{k,n}\}_{k\geq 0} ,\]
	\[ \spec(\Lap^{S^n}_1) =  \{\lambda^0_{k,n} -1\}_{k\geq 1}  \cup  \{\lambda^1_{k,n} -1\}_{k\geq 1} ,\]
	\[\spec(\Lap^{S^n}_E)  = \{\lambda^0_{k,n}-2\}_{k\geq 0} \cup  \{\lambda^1_{k,n}-2\}_{k\geq 2} \cup  \{\lambda^2_{k,n}-2\}_{k\geq 2} ,\]
In particular, since $\lambda^2_{2,n}= \frac{4n}{n-1}>2$ for any $n\geq2$, there are no TT eigentensors of eigenvalue 0, that is, $\ker(\Lap_L+2)=0$.

Later we focus on the $n=2$ case (which is in fact the unit sphere), so we record separately \[  \{\lambda^0_{k,2}\}_{k\geq 0}  = \{ 0,2,6 \cdots\},\qquad  \{\lambda^1_{k,2}\}_{k\geq 1}  = \{2,6,\cdots\}, \qquad  \{\lambda^2_{k,2}\}_{k\geq 2}  =\{ 8,\cdots\}.\]

On  $S^2$, the $2$-eigenspace $\ker(\Lap_0+2)$ consists of the coordinate functions $\theta_i = \langle x,e_i\rangle$, which satisfy $\nabla \theta_i = e_i^T$, $\langle \nabla \theta_i, \nabla \theta_j\rangle = \delta_{ij}-\theta_i \theta_j$, $\Hess \theta_i = -\theta_i g$ and $\Lap_0(\theta_i \theta_j) = 2\delta_{ij}-6\theta_i \theta_j$. 

\subsection{Jacobi deformations on product manifolds}
\label{sec:jacobi-product}
If the product manifold is $S^m\times S^n$ for $m,n\geq 3$, there is no Jacobi deformation other than those induced by diffeomorphisms from the spectrum on product manifolds. This is a consequence of the following lemma.

\begin{lemma}
Let $M=S^m \times S^n$, with the round 1-Einstein metric $g$. If $m,n\geq 3$ then
\[\ker(\Lap_0+2)=0, \qquad \ker(\Lap_L+2)=0.\]
\end{lemma}
\begin{proof}
For the Laplacian on functions, note that for $n\geq 3$ we have $\lambda^0_{1,n} >1$, so the sum of any two nonzero eigenvalues is strictly greater than 2. But $\lambda^0_{2,n}>2$, so $\lambda_1(\Lap^M_0) >2$. 

For the Lichnerowicz Laplacian, note that $\spec(\Lap^{S^n}_1)$ is strictly positive, so clearly $0\notin \spec(\Lap^{S^m}_1) + \spec(\Lap^{S^n}_1)$. Also note that the least eigenvalue of $\Lap_E^{S^m}$ is $-2$, and the next eigenvalue is $>-1$. Since $\lambda^0_{1,n}>1$ and $\lambda^0_{2,n}>2$ we conclude using Theorem \ref{thm:spec-prod} that $0\notin \spec(\Lap^M_E)$. 
\end{proof}

In the following we only focus on the product manifolds with $S^2$ components. 

We give precise descriptions of Jacobi deformations on $S^2\times S^2$ and $S^2\times N$. Recall that we decomposed the kernel as \begin{equation}\mathcal{K} = (\mathcal{K}_0 + \mathcal{K}_1 )\oplus \ker (\Lap_L+2),\end{equation} 
where $\mathcal{K}_0$ is the space of Jacobi deformations from symmetries, and \begin{equation}\mathcal{K}_1:= \{ vg | v\in\ker(\Lap+2)\}.\end{equation}

Let $\pi_1,\pi_2$ be the projections onto each factor. 

\begin{lemma}
	Let $M=S^2\times S^2$ with the round 1-Einstein metric $g$, we have 
	\[ \ker(\Lap+2) = \{ \pi_1^* u_1 + \pi_2^* u_2 | u_1,u_2\in \ker(\Lap^{S^2}+2)\},\qquad \ker(\Lap_L+2)=0.\]
	\end{lemma}

\begin{proof}
For the Laplacian on functions, since $\lambda^0_{1,2} =2$, the only product eigenfunctions with eigenvalue 2 are those which are constant on one factor and have eigenvalue 2 on the other. 

For the Lichnerowicz Laplacian, $\spec(\Lap^{S^2}_1)$ is strictly positive, so clearly $0\notin \spec(\Lap^{S^2}_1) + \spec(\Lap^{S^2}_1)$. Also the least eigenvalue of $\Lap_E^{S^2}$ is $-2$, and the next eigenvalue is $0$. Moreover the eigenspaces are $\ker(\Lap^{S^2}_E) = \{ vg_{S^2} | v\in\ker(\Lap^{S^2}_0+2)\}$ and $\ker(\Lap^{S^2}_E-2) = \mathbb{R}g_{S^2}$. So the only product eigentensors in $\ker(\Lap^M_E)$ are of the forms $1\cdot v_2g^2$, $v_1 \cdot g^2$, $v_1g^1 \cdot 1$ or $g^1 \cdot v_2$. Here $g^b$ is the (pulled-back) metric on the $b$th $S^2$ factor, $b=1,2$, and $v_b = \pi_b^* u_b$, $u_b \in \ker(\Lap^{S^2}+2)$. 
But $v_b(g^1+g^2) = v_b g\in C^\infty(M)g$, and $v_b g^b= -\Hess v_b \in \im \mathcal{L}$. The TT kernel is therefore $\ker(\Lap^M_L +2)=0$. 
\end{proof}

We will consider 1-Einstein manifolds $N$, which additionally satisfy: 

\begin{itemize}
\item[($\dagger$)] $\lambda_1(\Lap^N_0) >2$, and $\ker(\Lap^N_L+2)=0$. 
\end{itemize}

\begin{lemma}
	Let $M=S^2\times N$, where $S^2$ has the round 1-Einstein metric and $N$ is a $1$-Einstein manifold satisfying ($\dagger$), we have
	\[ \ker(\Lap+2) = \{\pi_1^*u_1 | u_1\in \ker(\Lap^{S^2}+2)\},\qquad \ker(\Lap_L+2)=0.\]
\end{lemma}

\begin{proof}
For the Laplacian on functions, since $\lambda^0_{1,2}=2$, the only product eigenfunctions with eigenvalue 2 are those which are constant on $N$ and have eigenvalue 2 on the $S^2$ factor.

For the Lichnerowicz Laplacian, $\spec(\Lap^{S^2}_1)$ is strictly positive and $\spec(\Lap^N_1)$ is nonnegative for any Einstein manifold (cf. \cite[Lemma 4.4]{Kr4}),  so $0\notin \spec(\Lap^{S^2}_1) + \spec(\Lap^{N}_1)$. Then by $(\dagger)$, the only product eigentensors in $\ker(\Lap^M_E)$ are of the forms $1\cdot \pi_2^*(\mathcal{L}_N\omega)$ or $v g_{S^2} \cdot 1$, where $\omega \in \ker(\Lap^N_1+1)$ and $v = \pi_1^*u$, $u\in \ker(\Lap^{S^2}_0+2)$. But $\pi_2^*(\mathcal{L}_N\omega) = \mathcal{L}_M (\pi_2^*\omega)$, and again $v g_{S^2} = -\Hess v \in \im \mathcal{L}_M$. The TT kernel is therefore $\ker(\Lap^M_L+2)=0$. 
\end{proof}

Later, it will be convenient to collect the two previous cases as $M=\prod_{b=1}^B S^2_b \times \tilde{N}$. We will write any $v\in \ker(\Lap+2)$ as $v=\sum v_b$, where $v_b \in \ker(\Lap^b+2)$ and $\Lap^b$ is the Laplacian on the $b$th $S^2$ factor. It will be understood that each $v_b$ is independent of any other factors; explicitly, $v_b = \pi_b^* u_b$ in the above.

\section{Variations}
\label{sec:shrinker-variation}

In this section we compute the variation of the quantities $\nu, f, \tau, \Phi$, by considering certain families of metrics $g(s,t)$ with $g(0,0)=g$. The initial soliton will be an Einstein manifold $(M,g)$, with Einstein constant $\mu=1$, hence $\tau=\tau_0= \frac{1}{2}$ and $f=f_0= \frac{n}{2}+\nu$. In particular the initial weight $w=w_0$ is also a constant. 

Note that the variations of more familiar differential geometric quantities are given in Appendix \ref{sec:basic-var}. For the more complicated quantities here, the procedure will be as follows: 

\begin{enumerate}
\item Use lower order variations to calculate variation of $\nu$. 
\item Differentiate (\ref{eq:EL-nu2}) and $\int w\d V_g=1$ to calculate variation of $\tau$. 
\item Differentiate (\ref{eq:EL-nu}) to calculate variation of $f$. 
\item Substitute in to find variation of $\Phi$. 
\end{enumerate}

We always denote $u=(\Lap+ 1)\tilde{u}$ where $\int u \d V_g =\int \tilde{u} \d V_g=0$, and similarly for $v$.  

Note that any variation of $\Hess$ or $\Lap$ will act as 0 on any constant function on $M$. In particular, at the initial metric $f=f_0$ is constant so $\mathcal{D}^k\Hess \cdot f =0, \mathcal{D}^k \Lap \cdot f=0$, $k\geq 1$. 

For convenience, we define 
\[ \eta = \log (w\d V_g),\]
where $\d V_g$ understood as the volume ratio using the initial metric as the reference metric. 

Step (2) is a general procedure, but here as our initial soliton is Einstein, it will be expedient for us to actually consider $\int \phi w\d V_g$, where $\phi = (1-f_0) + f$. Note that $\int \phi w\d V_g = (1-f_0) + \frac{n}{2}+\nu$, so variations of $\int \phi w\d V_g$ match those of $\nu$. Moreover, $\phi_0 =1$, while all variations of $\phi$ match those of $f$. 

\subsection{First variation}
\label{sec:1st-variation}

We begin with the first variation as it is distinguished in our method. So we consider the 1-parameter family of metrics $g(t)$, where to start $g'$ is a general element of $\mathcal{S}^2(M)$ satisfying $\int (\tr_g g') \d V_g=0$. Here we use primes to denote the derivative in $t$, evaluated at 0.

As $\Phi=0$ at $t=0$, we immediately have \[\mathcal{D}\nu(g') =\nu' = -\int \langle \Phi, g'\rangle w\d V_g=0.\] 

Note that $\eta' = (\log(w\d V_g))' = -\frac{n\tau'}{2\tau}-f'+\frac{1}{2}\tr_g g'$. Then $ (\int \phi w\d V_g)'  = \int w\d V_g ( \phi \eta'+\phi') .$

Taking $\phi = (1-f_0) + f$, since $\nu'=0$ we have 
$0= \int w\d V_g(\eta' +f' )=\int w\d V_g(- \frac{n\tau'}{2\tau} + \frac{1}{2} \tr_g g').$ Therefore \[\mathcal{D}\tau(g')=\tau'=0.\]

Differentiating (\ref{eq:EL-nu}) and using $\nu'=\tau'=0$ gives
\begin{equation}\label{eq:1st-f} 0= -f' - \Lap f' - \frac{1}{2} R' .\end{equation} 

\subsubsection{TT direction}
\label{sec:1var-tt}

Here we compute $\mathcal{D}\Phi(h)$, where $h\in \mathring{\mathcal{S}}^2_g$ is transverse traceless. To match the notation used for higher variations, we use subscripts to denote derivatives in $t$. 

As $\tr h = \div h=0$, the variation formulae in Appendix \ref{sec:basic-1st} give $R_t =0$ and therefore \[\mathcal{D}f(h) = f_t =0.\] 

They also give \[\mathcal{D}\Ric(h) =\Ric_t = -R_{ikjl}h_{kl} + h_{ij} - \frac{1}{2}\Lap h_{ij} = -\frac{1}{2} \Lap_L h,\] where $\Lap_L$ is the Lichnerowicz Laplacian. 

Then \[\mathcal{D}\Phi(h) = \Phi_t = \tau \Ric_t - \frac{g_t}{2} = -\frac{1}{4}(\Lap_L h + 2h).\] 

\subsubsection{Conformal direction}
\label{sec:1var-conf}

Here we compute $\mathcal{D}\Phi(vg)$, where $v=(\Lap+1)\tilde{v}$, $\int\tilde{v}\d V_g=0$. For use in higher variations, it will be convenient to change the parameter to $s$. That is, here we consider the family $(1+sv)g$. 

Since $\Lap$ commutes with itself, (\ref{eq:1st-f}) gives that \[ \mathcal{D}f(vg) =f_s =-(\Lap+1)^{-1} \frac{R_s}{2}= \left(\frac{n}{2} + \frac{n-1}{2}\Lap\right) \tilde{v} = \frac{1}{2} \tilde{v} + \frac{n-1}{2} v.\] 

Here we have used the variation formulae in Appendix \ref{sec:basic-1st} (or Appendix \ref{sec:basic-conf}). It will be useful for later to record \[\mathcal{D}\Ric(vg) =\Ric_s=  -\frac{n-2}{2}\Hess v - \frac{1}{2}(\Lap v) g. \]

Then we have 
\begin{align*} \mathcal{D}\Phi(vg)=\Phi_s &=&& \tau (\Ric_s  + \Hess f_s) -\frac{g_s}{2} 
\\&=&& \frac{1}{2\mu}\left(  -\frac{n-2}{2}\Hess v - \frac{1}{2}(\Lap v) g + 0 + \frac{1}{2}\Hess ( \tilde{v} +(n-1)v)\right) - \frac{vg}{2}
\\&=&& -\frac{1}{4}(\Lap+2)v g +   \frac{1}{4}\Hess (\Lap+2) \tilde{v}. 
\end{align*}

This identifies the conformal kernel as $\{vg| v\in \ker(\Lap+2)\}$. 

\subsection{Higher variations in the conformal kernel direction}
\label{sec:ker-variation}

We now proceed to find variations in the directions $vg$, where $v\in \ker(\Lap+2)$, by consider the 1-parameter family of metrics $(1+sv)g$. 

%
%
%
%
%
%
%
%


\subsubsection{Second variation}
\label{sec:2var-ker-ker}


First, we have $\nu_{ss} = -\int w\d V_g( \langle \Phi_s,g_s\rangle + \langle \Phi, g_{ss}\rangle + \langle \Phi,g_s\rangle \eta_s)$. Since at $s=0$ we have $\Phi=\Phi_s=0$, this immediately implies 
\[\mathcal{D}^2\nu(vg,vg) = \nu_{ss} = 0.\]

We calculate
\[\left(\int \phi w\right)_{ss} = \int w\d V_g ( \phi \eta_s^2 + \phi \eta_{ss} + 2\phi_s \eta_s + \phi_{ss}).\] 
By the above, since $\Lap v=-2v$, we have $f_s = \frac{n-2}{2}v$ and \[\eta_s = (\log(w\d V_g))_s = -\frac{n\tau_s}{2\tau}-f_s+\frac{1}{2}\tr_g g_s = v.\] We also compute \[\eta_{ss} = -\frac{n\tau_{ss} }{2\tau} + \frac{n(\tau_s )^2}{2\tau^2} - f_{ss} + \frac{1}{2}\tr_g g_{ss} - \frac{1}{2}|g_s|^2 = -\frac{n\tau_{ss}}{2\tau}  - f_{ss} -\frac{n}{2}v^2.\] Then once again taking $\phi=(1-f_0)+f$, we have
\[ 0 = \int w\d V_g(\eta_s^2+\eta_{ss} + 2 f_s \eta_s + f_{ss})=\int w\d V_g \left(v^2+ (n-2)v^2 -\frac{n\tau_{ss}}{2\tau} - \frac{n}{2}v^2\right) .\]
This gives
 \[\mathcal{D}^2 \tau (vg,vg)=\tau_{ss} = \frac{n-2}{2n}\int v^2 w\d V_g,\] 

Now differentiating (\ref{eq:EL-nu}) twice, we have
\[ 0=-\tau_{ss} R -f_{ss} + \frac{1}{2}( -2\Lap f_{ss} - 4\Lap_s f_s + 2|df_s|^2 - R_{ss}).\] Using the conformal variation formulae in Appendix \ref{sec:basic-conf} gives \begin{equation}\label{eq:f_ss}\mathcal{D}^2f(vg,vg)=f_{ss} = -n\tau_{ss} + (1+\Lap)^{-1}\left( nv^2 -\frac{3n-2}{4}|dv|^2\right)  . \end{equation}

At $s=0$ we then calculate $\Phi_{ss} =\tau_{ss}(\Ric+\Hess f) + \tau( \Ric_{ss} + \Hess f_{ss} + 2\Hess_s f_s)$, so using $\Phi=0$ we have \begin{equation}\label{eq:2var-ker-ker}\mathcal{D}^2\Phi(vg,vg)=\Phi_{ss} =\frac{\tau_{ss}}{2}g+ \frac{1}{2}\Hess f_{ss}+  \frac{n-2}{2} v\Hess v +\frac{n-2}{4} dv\otimes dv +\frac{1}{2}(-2v^2 +|dv|^2)g. \end{equation} 

Note that, by differentiating (\ref{eq:divPhi}) twice, at $s=0$ (using $\Phi=\Phi_s=0$) we have 
\begin{equation} \div \mathcal{D}^2\Phi(vg,vg) = \div \Phi_{ss}=0.\end{equation}

\subsubsection{Third variation}
\label{sec:third-variation}

Here we compute $\langle \mathcal{D}^3\Phi(vg,vg,vg) , vg\rangle_{L^2}$, where $v\in \ker(\Lap+2)$. 

For brevity, we will omit the solution of $\mathcal{D}^3\tau(vg,vg,vg)=\tau_{sss}$ (and $\nu_{sss}$) - it is constant on $M$, and therefore will not contribute after integration against $v$ anyway. 

The third derivative of (\ref{eq:EL-nu}) gives 
\[\begin{split} 0=& -\tau_{sss} R - f_{sss} + 3\tau_{ss}(-2\Lap f_s-R_s) \\&+ \tau\left(-2\Lap f_{sss}-6\Lap_s f_{ss}-6\Lap_{ss} f_s-R_{sss} + 6 (g^{ij})_s (f_s)_{,i} (f_s)_{,j} +6g^{ij}(f_s)_{,i}(f_{ss})_j\right).\end{split}\] 

Using the variation formulae in Appendix \ref{sec:basic-conf} then gives the following defining equation for $\mathcal{D}^3f(vg,vg,vg)=f_{sss}$, 
\begin{equation}\begin{split}\label{eq:f_sss} (1+\Lap)f_{sss} + n\tau_{sss} = &-3(n-2) \tau_{ss} v  + 3v\Lap f_{ss}  \\& -3v^3(3n-2) - \frac{12n^2-75n+66}{4} v|dv|^2.\end{split}\end{equation}

We then have \[\Phi_{sss} =\tau_{sss}(\Ric+\Hess f) + 3\tau_{ss}(\Ric_s + \Hess f_s)+ \tau(\Ric_{sss} + 3\Hess_{ss} f_s + 3\Hess_s f_{ss} +\Hess f_{sss}).\] Since $\Phi=\Phi_s=0$ at $s=0$, we then have
\begin{equation}
\begin{split} \mathcal{D}^3\Phi(vg,vg,vg)  =& \frac{1}{2}\Hess f_{sss} + \tau_{sss} g + 3\tau_{ss} vg  - \frac{3(n-2)}{2} v^2 \Hess v\\& - 3(n-2) v dv\otimes dv + 3v^3 g - \frac{3(3n-10)}{4} v|dv|^2 g \\&+ \frac{3}{4} (-dv\otimes df_{ss} - df_{ss} \otimes dv+ \langle dv,df_{ss}\rangle g).\end{split}\end{equation}

Note that the $\tau_{sss}g$ term will vanish upon integration against $vg$.

Taking the inner product we have \begin{equation}\begin{split}\label{eq:3rd-var}\langle  \mathcal{D}^3\Phi(vg,vg,vg) &, vg\rangle_{L^2} \\& = \int \left( 6(n-1)v^4 -\frac{9n^2-18n-24}{4} v^2 |dv|^2 + \frac{3(n-2)}{4} v\langle dv,df_{ss}\rangle\right)w\d V_g \\&\quad+ \frac{1}{2}\int v\Lap f_{sss} w\d V_g +\frac{3(n-2)}{2}\left(\int v^2w\d V_g\right)^2 .\end{split}\end{equation}

Recall $f_{ss}$ and $f_{sss}$ are defined by (\ref{eq:f_ss}) and (\ref{eq:f_sss}) respectively. 

\subsection{Second variation cross terms}
Finally, we calculate certain cross terms for $\mathcal{D}^2\Phi$. 

\subsubsection{Kernel-TT direction}
\label{sec:2var-ker-tt}

Here we compute $\langle \mathcal{D}^2\Phi(vg,h), ug\rangle_{L^2}$, where $v\in \ker(\Lap+2)$, $h\in \mathcal{\mathring{S}}^2_g$ and $u$ is any smooth function on $M$, by considering the 2-parameter variation $g+svg+th$. 

By the calculations in Section \ref{sec:1st-variation} we have $\tau_s=\tau_t=0$ and $f_s = \frac{n-2}{2}v$, $f_t=0$. 

First, using $\Phi=0$ at $(s,t)=(0,0)$, we have \[ \mathcal{D}^2\nu(vg,h)=\nu_{st} = -\int \langle \Phi_s,g_t\rangle w\d V_g = -\int v \tr_g(\Phi_t)  w\d V_g=0,\] as the Lichnerowicz Laplacian preserves the TT decomposition. 

Calculate \[\left(\int \phi w\right)_{st} = \int w \d V_g (\phi\eta_s\eta_t + \phi_s\eta_t + \phi_t\eta_s + \phi\eta_{st}+\phi_{st}).\] Now as $\langle g_s,g_t\rangle = v\langle g,h\rangle =0$ we have $\eta_{st}= -\frac{n}{2} \frac{\tau_{st}}{\tau} - f_{st}$, and from the calculations in Section \ref{sec:1st-variation} we have $\eta_s =v$, $\eta_t=0$. Taking $\phi = 1+f-f_0$ again, we have $\phi_t=f_t=0$ and $\phi_{st}=0$. So since $\nu_{st}=0$ we have \[ 0= \int w \d V_g (\eta_s\eta_t + f_s\eta_t + f_t\eta_s + \eta_{st}+f_{st})\] and hence \[\mathcal{D}^2\tau (vg,h) = \tau_{st} =0.\] 

Differentiating (\ref{eq:EL-nu}) in $s$ then $t$ gives $\tau(-2\Lap_t f_s - 2\Lap f_{st} - R_{st}) -f_{st}=0.$

Note that $\Lap_t$ acts as $\Lap_t \phi = - \langle h, \Hess \phi\rangle$. Using the formulae in Section \ref{sec:basic-ker-TT} then gives $(\Lap+1)f_{st} = \frac{n-2}{2}\langle h, \Hess v\rangle -\frac{1}{2}R_{st} =  0,$ hence 
\[\mathcal{D}^2f(vg,h) = f_{st}=0.\] 

As $g_{st}=0$ we then have $\Phi_{st} = \tau (\Ric_{st} +\Hess_t f_s).$ Since $h$ is TT, the first variation Lemma \ref{lem:basic-1st} gives $\Hess_t=0$. Therefore $\langle \Phi_{st}, ug\rangle_{L^2} = \frac{1}{2}\langle \Ric_{st}, ug\rangle_{L^2}$. 

Again using Appendix \ref{sec:basic-ker-TT} we conclude that \begin{equation}\label{eq:2var-ker-tt}\langle\mathcal{D}^2\Phi(vg,h), ug\rangle_{L^2} = \frac{1}{2} \langle \mathcal{D}^2\Ric(vg,h),ug\rangle_{L^2} =  \frac{n-2}{4} \int u\langle h,\Hess v\rangle w\d V_g.\end{equation}

\subsubsection{Kernel-conformal direction}
\label{sec:2var-ker-conf}

Here we compute $\langle \mathcal{D}^2 \Phi(vg,ug),vg\rangle_{L^2}$, where $v\in \ker(\Lap+2)$, $u= (\Lap+1)\tilde{u}$ and $\int u \d V_g=0$, by considering the metric variation $(1+sv+tu)g$. We furthermore assume $\tilde{u}$ (hence $u$) is orthogonal to $\ker (\Lap+2)$. By calculations in Section \ref{sec:1st-variation} we have $\tau_s =\tau_t = 0$ and $f_s = \frac{n-2}{2}v$, $f_t = \frac{1}{2}\tilde{u} + \frac{n-1}{2}u$. 

Again since $\Phi_s=0$ at $s=0$, we have \[ \mathcal{D}^2\nu(vg,ug)=\nu_{st} = -\int \tau \langle \Phi_s ,g_t\rangle w\d V_g=0.\]

As before we have \[\left(\int \phi w\right)_{st} = \int w \d V_g (\phi\eta_s\eta_t + \phi_s\eta_t + \phi_t\eta_s + \phi\eta_{st}+\phi_{st}).\]
Now as $\langle g_s,g_t\rangle = nuv$ we have $\eta_{st}= -\frac{n}{2} \frac{\tau_{st}}{\tau} - f_{st}-\frac{1}{2}nuv$, and as in Section \ref{sec:1st-variation} we have $\eta_s=v$, $\eta_t = \frac{1}{2}(u-\tilde{u})$. Taking $\phi = 1+f-f_0$ again, since $\nu_{st}=0$ we again have \[0=\int w \d V_g (\eta_s\eta_t + f_s\eta_t + f_t\eta_s + \eta_{st}+f_{st}).\] It then follows from $u,\tilde{u}$ being orthogonal to $v\in \ker(\Lap+2)$ that \[\mathcal{D}^2\tau (vg,ug) = \tau_{st} =0.\] 

Differentiating (\ref{eq:EL-nu}) in $s$ then $t$ gives $\tau(-2\Lap_s f_t - 2\Lap_t f_s - 2\Lap f_{st} - R_{st})-f_{st}=0$. This yields the defining equation for $\mathcal{D}^2f(vg,ug)= f_{st}$, 

\begin{equation}\label{eq:f-ker-conf}(\Lap+1)f_{st} = \frac{1}{2}v\Lap \tilde{u} - \frac{n-1}{2}v\Lap u -\frac{n-2}{4}\langle d\tilde{u},dv\rangle  -\frac{n^2}{4} \langle du,dv\rangle.\end{equation}

Then, using the conformal cross-variations in Appendix \ref{sec:conf-var-cross}, we have  \begin{align*}\mathcal{D}^2\Phi(vg,ug)=\Phi_{st} &=&& \tau(\Ric_{st} + \Hess_s f_t + \Hess_t f_s + \Hess f_{st}) \\&=&& \frac{3(n-2)}{8}(du\otimes dv+dv\otimes du) + \frac{n-2}{4}(u\Hess v + v\Hess u ) \\&&&+ \frac{1}{4}(-2uv + v\Lap u)g - \frac{n-4}{4}\langle du,dv\rangle g \\&&&+ \frac{1}{4}(-dv\otimes f_t + df_t\otimes dv + \langle dv,df_t\rangle g) \\&&&+ \frac{n-2}{8}(-du\otimes dv - dv\otimes du + \langle du,dv\rangle g) + \frac{1}{2}\Hess f_{st}.\end{align*} 

Taking the inner product with $vg$ and collecting terms, it follows that \begin{equation}\begin{split}\label{eq:2var-ker-conf}\langle \mathcal{D}^2\Phi(vg,ug) ,vg\rangle_{L^2} =& \int \left( \frac{n^2}{4}v\langle du,dv\rangle -(n-1)uv^2 + \frac{n-2}{4}v\langle d\tilde{u},dv\rangle + \frac{n-1}{2}v^2\Lap u\right) w\d V_g \\&+\frac{1}{2}\int v\Lap f_{st} w\d V_g,\end{split}\end{equation}
where $f_{st}$ is defined as in (\ref{eq:f-ker-conf}).

\section{Specialising to $S^2$ products}
\label{sec:S2-special}

In this section we consider a product of 1-Einstein manifolds $M= S^2\times N$, where $S^2$ has the (unit) round metric. The goal is to derive a formal obstruction for $\Phi$ at order 3. To do so, we will use explicit knowledge of the Jacobi deformations on $M$ together with the formulae in the previous section. 

\subsubsection{Explicit form}
\label{sec:kernel-explicit}

We assume that $N$ either is $S^2$, or satisfies assumption ($\dagger$), so that the results of Section \ref{sec:jacobi-product} apply. It will actually be convenient to unify the cases by collecting $S^2$ factors: We write $M= \prod_{b=1}^B S^2_b \times \tilde{N}$, where $B=2,\tilde{N} = *$ if $N=S^2$, and otherwise $B=1, \tilde{N}=N$. Note here the subscript $b$ serves only to distinguish the isometric $S^2$ factors. 

Throughout this section, we consider $v\in \ker(\Lap+2)$. As in Section \ref{sec:jacobi-product}, we have $v= \sum_b v_b$, where we may write $v_b = \alpha_b \theta^b$ and $\theta^b = \langle x,y_b\rangle$ for some $\alpha_b\in\mathbb{R}$ and some vector $|y_b|=1$. Then as in Section \ref{sec:sph-spec}, we have $|dv_b|^2 = \alpha_b^2 - v_b^2$ and $\Lap v_b^2 = 2\alpha_b^2 - 6v_b^2$. 

We introduce the notation $\sigma_k = \sum_b \alpha_b^k$ and $S_v = \sum v_b^2$. Then $v^2 = \sum \alpha_a \alpha_b  \theta^a \theta^b$, $|dv|^2 = \sigma_2 - S_v$, and $\Lap(v^2) = 2\sigma_2 - 4v^2 - 2S_v$. Finally note $\Hess v = -\sum v_b g^b$. \\

\subsection{Spherical integrals}

\begin{lemma}
\label{lem:sph-int}
Let $M= \prod_{b=1}^B S^2_b \times \tilde{N}$ and $v=\sum v_b$ be as in Section \ref{sec:kernel-explicit}. Then

\[w\int \d V_g=1,, \qquad w\int v_b^2\d V_g = \alpha_b^2/3,\]
\[w\int v_b^4\d V_g= \alpha_b^4/5, \qquad w\int v_a^2 v_b^2\d V_g = \alpha_a^2\alpha_b^2/9 \quad\text{  if  } a\neq b.\]
Moreover, 
\[w\int v^2\d V_g = w\int S_v \d V_g = \frac{1}{3}\sigma_2, \qquad w\int v^2 v_b^2\d V_g = \frac{1}{9} \sigma_2 \alpha_b^2 + \frac{4}{45} \alpha_b^4,\]

\[w\int v^4\d V_g = \frac{1}{3} \sigma_2^2 -\frac{2}{15}\sigma_4,\qquad w\int S_v^2\d V_g = \frac{1}{9}\sigma_2^2 + \frac{4}{45}\sigma_4.\]
\end{lemma}
\begin{proof}
Note that $\int_{S^2} \theta_j^2 = \frac{4\pi}{3}$, $\int_{S^2}\theta_j^4 =\frac{4\pi}{5}$ and $|S^2|=4\pi$. Also note that $w^{-1} = |M| = (4\pi)^B |\tilde{N}|$. The first four equations are then routine computations. For the last four, note that spherical polynomials of odd degree integrate to 0. 
\end{proof}

\subsection{Explicit solutions}

In this subsection, we use the explicit form of $v$ to explicitly solve for certain quantities in Section \ref{sec:shrinker-variation}. 

We will frequently use the ansatz that a function $\phi$ is a linear combination \begin{equation}\phi= A_1 v^2 + A_2 S_v + A_3 \sigma_2 + B_1 vv_b + B_2 v_b^2 + B_3 \alpha_b^2,\end{equation} and formally solve the resulting equations as linear systems. Under this ansatz, we represent such a function by the column vector $\begin{pmatrix}A_1\\A_2\\A_3\\B_1\\B_2\\B_3\end{pmatrix}$, and the Laplacian $\Lap$ acts as \[M= \begin{pmatrix} -4 & 0 &0&0&0&0 \\ -2&-6&0&0&0&0\\ 2&2&0&0&0&0\\ 0&0&0&-4&0&0\\ 0&0&0&-2&-6&0\\ 0&0&0&2&2&0\end{pmatrix}.\] The Laplacian $\Lap^b$ on a particular $S^2_b$ factor acts as \[M_b= \begin{pmatrix} 0&0&0&0&0&0\\0&0&0&0&0&0\\0&0&0&0&0&0\\-4&0&0&-2&0&0\\-2&-6&0&-4&-6&0\\2&2&0&2&2&0\end{pmatrix}.\] 

\subsubsection{Second variation of $f$}

Here we find an explicit solution for the function $f_{ss} = \mathcal{D}^2f(vg,vg)$ in Section \ref{sec:2var-ker-ker}. Recall that the defining equation (\ref{eq:f_ss}) is \[ (1+\Lap)f_{ss} = nv^2 -\frac{3n-2}{4}|dv|^2 -n\tau_{ss} = nv^2 + \frac{3n-2}{4} S_v -\frac{11n-10}{12}\sigma_2,\] 
where we have used \[\tau_{ss}=\mathcal{D}^2\tau(vg,vg) = \frac{n-2}{2n}\int v^2 w \d V_g = \frac{n-2}{6n}\sigma_2.\]
We use the ansatz $f_{ss}  = A_1 v^2 + A_2 S_v + A_3 \sigma_2$, and solve the linear system as \[\begin{pmatrix}A_1\\A_2\\A_3\\0\\0\\0\end{pmatrix} = (1+M)^{-1}\begin{pmatrix} n\\ \frac{3n-2}{4} \\ -\frac{11n-10}{12} \\0\\0\\0\end{pmatrix}=\begin{pmatrix} -\frac{n}{3}\\ -\frac{n-6}{60}  \\ -\frac{13n-38}{60} \\0\\0\\0\end{pmatrix}.\]

One may verify that this solution indeed satisfies the defining equation, which proves
\begin{lemma}
\label{lem:explicit-f_ss}
Suppose $M=\prod_{b=1}^B S^2_b \times \tilde{N}$ and $v$ is as in Section \ref{sec:kernel-explicit}. Then \begin{equation}\label{eq:explicit-f_ss} \mathcal{D}^2f(vg,vg) =f_{ss}= -\frac{n}{3} v^2  -\frac{n-6}{60} S_v -\frac{13n-38}{60}\sigma_2\end{equation} is the unique solution of (\ref{eq:f_ss}). 
\end{lemma}

\subsubsection{Second variation of $\Phi$}

\begin{proposition}
Let $M= \prod_{b=1}^B S^2_b \times \tilde{N}$ and $v=\sum v_b$ be as in Section \ref{sec:kernel-explicit}. Then \[\pi_{\mathcal{K}} \mathcal{D}^2\Phi(vg,vg) =0.\] 
\end{proposition} 
\begin{proof}
By the results of Section \ref{sec:jacobi-product}, we have $\mathcal{K} = \mathcal{K}_0+ \mathcal{K}_1$, and $\mathcal{K}_1 = \{vg | v\in\ker(\Lap+2)\}$. Since $\div\mathcal{D}^2\Phi(vg,vg)=0$, we immediately have $\pi_{\mathcal{K}_0} \mathcal{D}^2\Phi(vg,vg) =0$. 

Now the second variation formula (\ref{eq:2var-ker-ker}) and Lemma \ref{lem:explicit-f_ss}, we see that $\tr_g \mathcal{D}^2\Phi(vg,vg)$ is a spherical polynomial (on $\prod_{b=1}^B S^2_b$) of even (total) degree. Then for any $v'\in \ker(\Lap+2)$, we see that $\langle \mathcal{D}^2\Phi(vg,vg), v' g\rangle_{L^2}$ is the integral of a spherical polynomial of odd degree, and therefore vanishes. This implies $\pi_{\mathcal{K}_1} \mathcal{D}^2\Phi(vg,vg) =0$ and hence the result.
\end{proof}

\subsubsection{Second variation of the metric}\label{SSS:Second variation of the metric}

Since $\pi_\mathcal{K} \mathcal{D}^2\Phi(vg,vg)=0$, there is now a unique $\tilde{u} \in \mathring{C}^\infty_g$, $\tilde{u} \perp \ker(\Lap+2)$ and $h\in \mathring{\mathcal{S}}^2_g$ such that \begin{equation}\label{eq:second-order-metric}\mathcal{D}\Phi(ug+h) = -\frac{1}{2}\mathcal{D}^2\Phi(vg,vg),\end{equation} where $u=(\Lap+1)\tilde{u}$. We will give an explicit solution for $u$. For $h$, it will later suffice to give an explicit solution for $\langle h,g^b\rangle$ modulo $\ker(\Lap+2)$.

First, tracing (\ref{eq:second-order-metric}) using the first variation formulae in Section \ref{sec:1st-variation} and the second variation formulae (\ref{eq:2var-ker-ker}) gives the defining equation for $\tilde{u}$: 

\begin{equation}\label{eq:def-u-tilde}\begin{split} (\Lap+2)(n+(n-1)\Lap)\tilde{u}&= -4 \tr \mathcal{D}\Phi(ug)= 2\tr \mathcal{D}^2\Phi(vg,vg) \\&=  n\mathcal{D}^2\tau(vg,vg)+ \Lap \mathcal{D}^2 f(vg,vg) -4(n-1)v^2 + \frac{3n-2}{2}|dv|^2 \\&=\Lap \mathcal{D}^2 f(vg,vg) -4(n-1)v^2 -\frac{3n-2}{2} S_v +\frac{5n-4}{3}\sigma_2 .\end{split}\end{equation}

Using the same ansatz $\tilde{u}=A'_1 v^2 + A'_2 S_v + A'_3 \sigma_2$, and the explicit solution of $\mathcal{D}^2 f(vg,vg)$, we can solve the linear system as \begin{align*}\begin{pmatrix}A'_1\\A'_2\\A'_3\\0\\0\\0\end{pmatrix} &=&& (n+(n-1)M)^{-1}(2+M)^{-1} M\begin{pmatrix} -\frac{n}{3}\\ -\frac{n-6}{60}  \\ -\frac{13n-38}{60} \\0\\0\\0\end{pmatrix} \\&&&+  (n+(n-1)M)^{-1}(2+M)^{-1}\begin{pmatrix} -4(n-1)\\ -\frac{3n-2}{2} \\ \frac{5n-4}{3} \\0\\0\\0\end{pmatrix}\\&=&& \begin{pmatrix}   -\frac{2(2n-3)}{3(3n-4)}\\ \frac{247n^2-678n+456}{60(3n-4)(5n-6)} \\ -\frac{(n-6)(4n-5)}{30n(5n-6)} \\0\\0\\0 \end{pmatrix}. \end{align*}
 One may verify that the corresponding solution \[\tilde{u} = -\frac{2(2n-3)}{3(3n-4)}v^2 +\frac{247n^2-678n+456}{60(3n-4)(5n-6)} S_v -\frac{(n-6)(4n-5)}{30n(5n-6)}\sigma_2\]  indeed satisfies equation (\ref{eq:def-u-tilde}). For convenience we also record
\begin{equation}\label{eq:u-solution} u = (\Lap+1)\tilde{u} = \frac{2(2n-3)}{3n-4}v^2 - \frac{29n^2-82n+56}{4(3n-4)(5n-6)} S_v - \frac{11n^2-19n+6}{6n(5n-6)}\sigma_2.\end{equation}

With $\tilde{u}$ in hand, we consider the pointwise inner products $h_b := \langle h,g^b\rangle$. We have \[\langle \mathcal{D}\Phi (ug+h), g^b\rangle = \frac{1}{4}\Lap_b (\Lap+2)\tilde{u} - \frac{1}{2}(\Lap+2)u - \frac{1}{4}(\Lap+2) h_b,\] so pairing (\ref{eq:second-order-metric}) with $g^b$ we find the defining equation 
\begin{align}\label{eq:def-hb} (\Lap+2) h_b &=&&  (\Lap+2)( \Lap^b \tilde{u} - 2u)  +2\langle \mathcal{D}^2\Phi(vg,vg) , g^b\rangle \\ \nonumber &=&&  (\Lap+2)( \Lap^b \tilde{u} - 2u) +\frac{n-2}{3n}\sigma_2 + \Lap^b \mathcal{D}^2f(vg,vg) \\\nonumber &&& - 2(n-2) vv_b + \frac{n-2}{2} |dv_b|^2 - 4v^2 + 2|dv|^2.\end{align}

We will find $\tilde{h}_b$ so that $(\Lap+2)\tilde{h}_b$ equals the right hand side of (\ref{eq:def-hb}); then $h_b - \tilde{h}_b \in \ker(\Lap+2)$. Here we use the ansatz $\tilde{h}_b =A''_1 v^2 + A''_2 S_v + A''_3 \sigma_2 + B_1 vv_b + B_2 v_b^2 + B_3 \alpha_b^2$. Using the solution of $\tilde{u}$ (hence $u$), we may solve the linear system as
\begin{align*}\begin{pmatrix}A''_1\\A''_2\\A''_3\\B_1\\B_2\\B_3\end{pmatrix} &=&& M_b \begin{pmatrix}   -\frac{2(2n-3)}{3(3n-4)}\\ \frac{247n^2-678n+456}{60(3n-4)(5n-6)} \\ -\frac{(n-6)(4n-5)}{30n(5n-6)} \\0\\0\\0 \end{pmatrix} - 2\begin{pmatrix} \frac{2(2n-3)}{3n-4}\\ - \frac{29n^2-82n+56}{4(3n-4)(5n-6)}\\ - \frac{11n^2-19n+6}{6n(5n-6)} \\0\\0\\0 \end{pmatrix} 
\\&&& +(M+2)^{-1}M_b \begin{pmatrix} -\frac{n}{3}\\ -\frac{n-6}{60}  \\ -\frac{13n-38}{60} \\0\\0\\0\end{pmatrix}+ (M+2)^{-1} \begin{pmatrix} -4\\-2 \\ \frac{7n-2}{3n}\\ -2(n-2) \\  -\frac{n-2}{2}\\ \frac{n-2}{2} \end{pmatrix}
\\&=&& \begin{pmatrix} -\frac{2(n-2)}{3n-4}\\ \frac{(n-2)(7n-8)}{(3n-4)(5n-6)} \\ \frac{(n-2)(17n-18)}{6n(5n-6)} \\  \frac{n(n-2)}{3n-4} \\ -\frac{n(n-2)(7n-8)}{2(3n-4)(5n-6)} \\ -\frac{(n-1)(n-2)}{5n-6}. \end{pmatrix}.\end{align*}
One may verify that the corresponding solution indeed solves (\ref{eq:def-hb}). In conclusion, we have proven

\begin{lemma}
\label{lem:2nd-order-sol}
Let $M= \prod_{b=1}^B S^2_b \times \tilde{N}$ and $v=\sum v_b$ be as in Section \ref{sec:kernel-explicit}. There is a unique solution $(\tilde{u},h)$ of $\mathcal{D}\Phi(ug+h) = -\frac{1}{2}\mathcal{D}^2\Phi(vg,vg)$, where $u=(\Lap+1)\tilde{u}$, $\tilde{u} \in \mathring{C}^\infty_g$, $\tilde{u} \perp \ker(\Lap+2)$, $h\in \mathring{\mathcal{S}}^2_g$. Moreover they satisfy
\begin{equation}\label{eq:u-tilde}\tilde{u} = -\frac{2(2n-3)}{3(3n-4)}v^2 +\frac{247n^2-678n+456}{60(3n-4)(5n-6)} S_v -\frac{(n-6)(4n-5)}{30n(5n-6)}\sigma_2,\end{equation} 
\begin{equation}\label{eq:h-b}\begin{split}h_b = \langle h, g^b\rangle =& -\frac{2(n-2)}{3n-4}v^2 + \frac{(n-2)(7n-8)}{(3n-4)(5n-6)} S_v + \frac{(n-2)(17n-18)}{6n(5n-6)}\sigma_2 \\&+ \frac{n(n-2)}{3n-4} vv_b  -\frac{n(n-2)(7n-8)}{2(3n-4)(5n-6)} v_b^2  -\frac{(n-1)(n-2)}{5n-6} \alpha_b^2 \\& +\varphi_b,\end{split}\end{equation}
for some $\varphi_b \in \ker(\Lap+2)$. 
\end{lemma}

\subsection{Integration by parts}

We will use the following integration by parts tricks. The first equation below, in particular, is used for terms involving derivatives of $f_g$. 

\begin{lemma}
\label{lem:ibp}
Let $M= \prod_{b=1}^B S^2_b \times \tilde{N}$ and $v$ is as in Section \ref{sec:kernel-explicit}. Then for any smooth function $\phi$, we have

\[ \int v\Lap \phi \d V_g = 2\int v(1+\Lap)\phi \d V_g.\] 

\[ \int v^2 \Lap \phi \d V_g = \int \phi (2\sigma_2- 4v^2 - 2S_v)\d V_g,\] 

\[ \int v\langle d\phi,dv\rangle \d V_g = \int \phi(2v^2 - \sigma_2 +S_v) \d V_g.\] 
\end{lemma}

\subsection{Second variation cross terms}

\begin{proposition}
\label{prop:2var-cross}
Let $M= \prod_{b=1}^B S^2_b \times \tilde{N}$ and $v\in\ker(\Lap+2)$. Let $u=(\Lap+1)\tilde{u}, h$ be the unique solution of $\mathcal{D}\Phi(ug+h) = -\frac{1}{2}\mathcal{D}^2\Phi(vg,vg)$ as in Lemma \ref{lem:2nd-order-sol}. Then we have
\begin{equation}\begin{split} \langle \mathcal{D}^2\Phi(vg,ug) ,vg\rangle_{L^2} = & \frac{2235n^3-9866n^2+14364n-6888}{675(3n-4)(5n-6)} \sigma_4 \\& - \frac{87n^3-265n^2+186n+24}{108n(3n-4)} \sigma_2^2,\end{split}\end{equation}
\begin{equation}\begin{split}\langle\mathcal{D}^2\Phi(vg,h), vg\rangle_{L^2} =&  -\frac{(n-2)^2(41n^2+40n-104)}{180(3n-4)(5n-6)} \sigma_4 \\& + \frac{(n-2)^2(29n^3-59n+90n-72)}{72n(3n-4)(5n-6)}\sigma_2^2.
\end{split}\end{equation}
\end{proposition}
\begin{proof}
Pair the second variation formula (\ref{eq:2var-ker-conf}) with $vg$. Using Lemma \ref{lem:ibp} and the explicit forms (\ref{eq:explicit-f_ss}, \ref{eq:u-tilde}) of $f_{ss}$ and $\tilde{u}$, every term in $\langle \mathcal{D}^2\Phi(vg,ug) ,vg\rangle_{L^2}$ may be expanded as a product of the functions $v^2, S_v, \sigma_2$. All such products have integral computed in Lemma \ref{lem:sph-int}, and collecting terms gives the first equation.

By formula (\ref{eq:2var-ker-tt}), we have $\langle\mathcal{D}^2\Phi(vg,h), vg\rangle_{L^2} = -\frac{n-2}{4} \int (\sum vv_b h_b) w\d V_g$. Now note that for any $v'\in\ker(\Lap+2)$, the product $vv_bv'$ is a spherical polynomial of total degree 3, and hence $\int vv_b v' \d V_g=0$. Therefore we need only integrate against the explicit portion of the solution (\ref{eq:h-b}) for $h_b$ - that is, against $\tilde{h}_b$ in the proof of Lemma \ref{lem:2nd-order-sol}. 

Consider the basis functions $v^2,  S_v, \sigma_2 , vv_b ,  v_b^2 , \alpha_b^2$. Multiplying by $vv_b$ and summing over $b$ respectively give $v^4, v^2 S_v^2, \sigma_2 v^2, v^2S_v, v\sum v_b^3, \sigma_2 vv_b$. Only the last two do not have integral listed in Lemma \ref{lem:sph-int}. For these we again note that spherical polynomials of odd degree integrate to 0, which implies $\int v(\sum v_b^3)\d V_g= \sum \int v_b^4\d V_g$ and $ \int vv_b\d V_g = \int v_b^2\d V_g$. Collecting terms then gives the second equation. 
\end{proof}

\subsection{Third variation}\label{SS:Third variation}

\begin{proposition}
\label{prop:3var-obs}
Let $M= \prod_{b=1}^B S^2_b \times \tilde{N}$ and $v\in\ker(\Lap+2)$. Then 

\begin{equation}\begin{split} \langle \mathcal{D}^3\Phi(vg,vg,vg) ,vg\rangle_{L^2} = & \frac{457n^2-1852n+660}{900} \sigma_4 \\&-\frac{23n^3-95n^2+36n+12}{18n}\sigma_2^2.\end{split}\end{equation}

\end{proposition}
\begin{proof}
Again, using Lemma \ref{lem:ibp}, every term in formula (\ref{eq:3rd-var}) for $\langle \mathcal{D}^3\Phi(vg,vg,vg),vg\rangle_{L^2}$ may be expanded as a product of the functions $v^2, S_v, \sigma_2$, which have integral computed in Lemma \ref{lem:sph-int}. Collecting terms gives the result. 
\end{proof}

\subsection{Formal obstruction}

Combining the above formulae, we have

\begin{theorem}
\label{thm:formal-obstruction}
Let $M= \prod_{b=1}^B S^2_b \times \tilde{N}$ and $v\in\ker(\Lap+2)$. Let $u=(\Lap+1)\tilde{u}, h$ be the unique solution of $\mathcal{D}\Phi(ug+h) = -\frac{1}{2}\mathcal{D}^2\Phi(vg,vg)$. Then we have

\begin{equation}
\langle \mathcal{D}^3\Phi(vg,vg,vg) + 6\mathcal{D}^2\Phi(ug+h,vg) ,vg\rangle_{L^2}
= Q_4 \sigma_4 + Q_2 \sigma_2^2,
\end{equation}

where \[Q_4 = \frac{5625n^4-23546n^3+15316n^2+28104n-26784}{900(3n-4)(5n-6)},\] \[Q_2 = -\frac{603n^5-3203n^4+4384n^3+108n^2-3120n+1152}{36n(3n-4)(5n-6)}.\] 

\end{theorem}


\begin{corollary}
\label{cor:formal-obstruction}
Let $M= \prod_{b=1}^B S^2_b \times \tilde{N}$ and suppose $v\in \ker(\Lap+2)$. Let $u=(\Lap+1)\tilde{u}, h$ be the unique solution of $\mathcal{D}\Phi(ug+h) = -\frac{1}{2}\mathcal{D}^2\Phi(vg,vg)$. If $n=4$, or if $B=1$, then there exists $\delta>0$ such that
\begin{equation}\label{eq:formal-obstruction} \|\pi_{\mathcal{K}_1}( \mathcal{D}^3\Phi(vg,vg,vg) + 6\mathcal{D}^2\Phi(ug+h, vg))\|_{L^2} \geq \delta \|v\|_{L^2}^3.\end{equation}
\end{corollary} 
\begin{proof}
Recall that $\|v\|_{L^2}^2 = \int v^2 w\d V_g = \frac{\sigma_2}{3}$. 

First assume $n=4$. Then have explicitly $Q_4 = \frac{4121}{1575}$ and $Q_2= -\frac{535}{126}$. By the trivial inequality $\sigma_4 \leq \sigma_2^2$ we then have
\[ Q_4 \sigma_4 + Q_2 \sigma_2^2 \leq (Q_4+Q_2)\sigma_2^2 =  - \frac{1711}{1050}\sigma_2^2 <0 .\] 

Now assume $B=1$. Then $\sigma_4 =\sigma_2^2$, and 

\[Q_4+Q_2 = -\frac{1050n^4-4881n^3+3968n^2+2468n-2400}{300n(5n-6)}.\] Note that the numerator has no integer roots. 

Therefore, in either case, \[|\langle \mathcal{D}^3\Phi(vg,vg,vg) + 6\mathcal{D}^2\Phi(ug+h,vg) ,vg\rangle_{L^2}| \geq \delta \|v\|_{L^2}^4,\] which implies the result. 
\end{proof}

Note that one may replace $L^2$ in (\ref{eq:formal-obstruction}) with any norm since $\mathcal{K}_1$ is finite-dimensional.

\section{H\"{o}lder theory}
\label{sec:holder}

In this section, we use Taylor expansion to turn the formal obstruction into the quantitative rigidity Theorem \ref{thm:main3}. Throughout this section $(M,g)$ will denote the 1-Einstein manifold $S^2\times N$, where either $N=S^2$ or $N$ satisfies ($\dagger$). 

Recall that $\mathcal{K}$ is the space of Jacobi deformations, and $\cK_1=\{ vg | (\Lap +2\mu) v=0\}$ is the subspace of Jacobi deformations modulo symmetries. 
We use $T_1 * T_2$ to denote an unspecified contraction of the tensors $T_1,T_2$. We will freely use that H\"{o}lder norms behave well under contractions, i.e. $\|T_1 * T_2 \|_{C^{k,\alpha}} \leq \|T_1\|_{C^{k,\alpha}} \|T_2\|_{C^{k,\alpha}}$. 

\subsection{Taylor expansion}

In the sequel, we would like to apply Taylor expansion to $\Phi(s) := \Phi(g+sh)$. To do so, we need a priori bounds for $\mathcal{D}^k\Phi$ in terms of a finite number of derivatives of $g$. However, this is not as simple as in \cite{SZ}, since the explicit form of $\Phi$ depends not only on the 2-jet of $g$, but also on the implicitly defined $\tau_g$ and (the 2-jet of) $f_g$. 

The goal is to show that $\mathcal{D}^k\nu$, $\mathcal{D}^k\tau$, $\mathcal{D}^k f$ and $\mathcal{D}^k\Phi$ are uniformly bounded as $k$-linear operators on $C^{2,\alpha}(\mathcal{S}^2(M))$ with values in $\mathbb{R}$, $\mathbb{R}$, $C^{2,\alpha}(M)$ and $C^{0,\alpha}(\mathcal{S}^2(M))$ respectively. Explicitly, let 
\[ \|\mathcal{D}^k\nu_g\| := \sup_{ \|h_i\|_{C^{2,\alpha}} \leq 1} |\mathcal{D}^k\nu_g(h_1,\cdots,h_k)|,\] 
\[ \|\mathcal{D}^k\tau_g\| := \sup_{ \|h_i\|_{C^{2,\alpha}} \leq 1} |\mathcal{D}^k\tau_g(h_1,\cdots,h_k)|,\] 
\[ \|\mathcal{D}^k f_g\|_{C^{2,\alpha}} := \sup_{ \|h_i\|_{C^{2,\alpha}} \leq 1} \|\mathcal{D}^k f_g(h_1,\cdots,h_k)\|_{C^{2,\alpha}},\] 
\[ \|\mathcal{D}^k\Phi_g\|_{C^{0,\alpha}} := \sup_{ \|h_i\|_{C^{2,\alpha}} \leq 1} \|\mathcal{D}^k\Phi_g(h_1,\cdots,h_k)\|_{C^{0,\alpha}}.\] 

Then we have

\begin{lemma}
\label{lem:shrinker-quantities-bound}
Fix a reference soliton metric $\Phi(g_0)=0$. There exists $C_k, \epsilon>0$ such that if $\|g-g_0\|_{C^{2,\alpha}} <\epsilon$, then for any $k$ the variations at $g$ satisfy
\begin{equation}  \|\mathcal{D}^k\nu_g\| + \|\mathcal{D}^k\tau_g\|+\|\mathcal{D}^k f_g\|_{C^{2,\alpha}}+\|\mathcal{D}^k \Phi_g\|_{C^{0,\alpha}} \leq C_k ,\end{equation}
\end{lemma}

The proof of Lemma \ref{lem:shrinker-quantities-bound} is deferred to Appendix \ref{sec:shrinker-quantities-bound}, and uses an inductive process similar to Section \ref{sec:shrinker-variation}. We now proceed with our Taylor expansion:

\begin{lemma}
\label{lem:taylor-exp-holder}
Fix a metric $g$. There exist $C,\epsilon>0$ so that if $\|h\|_{C^{2,\alpha}} <\epsilon$, then
\begin{equation}
\left\| \Phi(g+h) - \Phi(g)-\sum_{j=1}^{k} \frac{1}{j!}\mathcal{D}^j \Phi(h,\cdots,h)\right\|_{C^{0,\alpha}} \leq C\|h\|_{C^{2,\alpha}}^{k+1}.
\end{equation}
\end{lemma}
\begin{proof}
Set $\Phi(s) = \Phi(g+sh)$, so that $\frac{d^k}{ds^k} \Phi(s) = (\mathcal{D}^k\Phi)_{g+sh}(h,\cdots,h)$. We then apply Taylor's theorem with remainder (in integral form) pointwise to $\Phi(s)$ and $\Phi(s)|_x - \Phi(s)|_y$ to get the result. Note that Lemma \ref{lem:shrinker-quantities-bound} gives uniform bounds on $\|\mathcal{D}^{k+1}\Phi\|_{C^{0,\alpha}}$ in a neighbourhood of $g$, which controls the remainder.
\end{proof}

\subsection{H\"{o}lder estimate}

Here we first prove a quantitative rigidity for deformations modulo symmetries as follows:

\begin{theorem}\label{thm:quantitative rigidity}
Let $M= S^2\times N$, where $N=S^2$ or $N$ satisfies ($\dagger$). 
	There exists $C,\eps_1>0$ such that if $h\in \mathring{\mathcal{S}}^2_g \oplus (\Lap+1)\mathring{C}^\infty_g\cdot g$ and $\|h\|_{C^{2,\alpha}}\leq \eps_1$, then 
	\begin{equation} \|h\|_{C^{2,\alpha}}^{3}\leq C\|\Phi(g+h)\|_{C^{0,\alpha}}.\end{equation} 
\end{theorem}

The proof of Theorem \ref{thm:quantitative rigidity} is somewhat formal and proceeds as in \cite{SZ}. We reproduce the strategy here in detail for those unfamiliar with the setting there. 

We will repeatedly use the following estimate, which takes advantage of the fact that $\cK_1$ is finite-dimensional. 
\begin{lemma}\label{lem:k1bound}
	For any $h\in \cS^2(M)$, we have
	\begin{equation}
		\|\pi_{\cK_1}(h)\|_{C^{2,\alpha}}\leq  C \|\pi_{\cK_1}(h)\|_{L^2} \leq C \|h\|_{L^2}\leq  C'' \|h\|_{C^{0,\alpha}}.
	\end{equation}
\end{lemma}

\begin{proof}
The first two inequalities are the equivalence of norms on a finite-dimensional space. 
	Pythagoras theorem in $L^2$ space implies that
	$\|\pi_{\cK_1}(h)\|_{L^2}\leq \|h\|_{L^2}.$
	The last inequality is just integration as $M$ has finite volume. 
	
\end{proof}

We also need the following elliptic Schauder estimate:

\begin{lemma}
\label{lem:elliptic-holder}
There exists $C$ so that if $h \in \mathcal{S}^2(M)$, then 
\[ \|h - \pi_\mathcal{K}(h)\|_{C^{2,\alpha}} \leq C \|Lh\|_{C^{0,\alpha}}.\]
\end{lemma}

In what follows we consider a soliton metric $\Phi(g)=0$, and recall that $L=\mathcal{D}\Phi$. Also, we assume $\|h\|_{C^{2,\alpha}}<\epsilon$ for small enough $\epsilon\ll 1$ so that $\|\Phi(g+h)\|_{C^{0,\alpha}} <1$. We will frequently use this to absorb terms of higher degree, and also freely use Young's inequality. 

\subsubsection{1st order analysis}

By Lemma \ref{lem:taylor-exp-holder} with $k=1$ we have
\[ \|\Phi(g+h)-\Phi(g)-\mathcal{D}\Phi(h)\|_{C^{0,\alpha}}\leq C\|h\|^2_{C^{2,\alpha}}.\]
We consider $h\in \mathring{\mathcal{S}}^2_g \oplus (\Lap+1)\mathring{C}^\infty_g\cdot g$, and decompose $h=h_1+k_1$, where  $k_1 \in\mathcal{K}_1$ and so $h_1 \in \mathring{\mathcal{S}}^2_g$. Note that indeed $k_1=vg=\pi_{\mathcal{K}}(h) = \pi_{\mathcal{K}_1}(h)$ by the assumption on $h$. The above then becomes 
\begin{equation}\label{eq:holder-1}\|\Phi(g+h)-Lh\|_{C^{0,\alpha}}\leq C\|h\|^2_{C^{2,\alpha}}. \end{equation}

This allows us to show that the residual $h_1$ is of higher order again:

\begin{lemma} \label{lem:h_1estimate}
	There exist $C,\epsilon>0$ so that if $\|h\|_{C^{2,\alpha}}<\epsilon$ and $h_1$ is as above, then
	\begin{equation}
		\|h\|_{C^{2,\alpha}} \leq C(\|h\|_{C^{0,\alpha}}+\|\Phi(g+h)\|_{C^{0,\alpha}}),
	\end{equation}
	
	\begin{equation}
		\|h_1\|_{C^{2,\alpha}}\leq C(\|h\|^2_{C^{0,\alpha}}+\|\Phi(g+h)\|_{C^{0,\alpha}}).
	\end{equation}
\end{lemma}
\begin{proof}
	By Lemma \ref{lem:elliptic-holder} and (\ref{eq:holder-1}), we have
		\begin{equation}\label{eq:h_1estimate-1}
		\|h_1\|_{C^{2,\alpha}}\leq C\|\mathcal{D}\Phi(h_1)\|_{C^{0,\alpha}}\leq C(\|h\|^2_{C^{2,\alpha}}+\|\Phi(g+h)\|_{C^{0,\alpha}}).
	\end{equation}
Now by the triangle inequality 
	\begin{equation}
		\|h\|_{C^{2,\alpha}} \leq \|h_1\|_{C^{2,\alpha}}+ \|v\|_{C^{2,\alpha}} \leq
		C(\|h\|_{C^{0,\alpha}}+\|h\|_{C^{2,\alpha}}^2+\|\Phi(g+h)\|_{C^{0,\alpha}}) .
	\end{equation}
	Absorbing the higher power of $\|h\|_{C^{2,\alpha}}<\epsilon \ll 1$ into the left hand side gives the first inequality. Substituting back in to (\ref{eq:h_1estimate-1}) and absorbing higher order terms again gives the second inequality. 
\end{proof}

%

\subsubsection{2nd order analysis}
By Lemma \ref{lem:taylor-exp-holder} with $k=2$ we have
\[\|\Phi(g+h)-\Phi(g)-\mathcal{D}\Phi(h)-\frac{1}{2}\mathcal{D}^2\Phi(h,h)\|_{C^{0,\alpha}}\leq C\|h\|^3_{C^{2,\alpha}}.\]
Expanding $\mathcal{D}^2\Phi$ by our decomposition of $h$ and using the uniform estimates on $\mathcal{D}^2\Phi$, 
\[\|\Phi(g+h)-Lh_1-\frac{1}{2}\mathcal{D}^2\Phi(vg,vg)
\|_{C^{0,\alpha}}\leq C(\|h\|^3_{C^{2,\alpha}}+\|h_1\|_{C^{2,\alpha}}^2 + \|h_1\|_{C^{2,\alpha}} \|v\|_{C^{2,\alpha}}).\]
As discussed in Section \ref{SSS:Second variation of the metric}, there is a unique solution $L(ug+k_2)=-\frac{1}{2}\mathcal{D}^2\Phi(vg,vg)$, where $u=(\Delta +1)\tilde{u}$, $\tilde{u} \in \mathring{C}^\infty_g$, $\tilde{u} \perp \ker(\Lap+2)$ and $k_2\in\mathring{\mathcal{S}}^2_g$. Then we have the further decomposition $h_1=ug+k_2+h_2$, and the above becomes
\begin{equation}\label{eq:holder-2}\| \Phi(g+h) - Lh_2 \|_{C^{0,\alpha}}
\leq C(\|h\|^3_{C^{2,\alpha}}+\|h_1\|_{C^{2,\alpha}}^2 + \|h_1\|_{C^{2,\alpha}} \|v\|_{C^{2,\alpha}}).\end{equation}

This allows us to show that the residual $h_2$ is of higher order:

\begin{lemma}
\label{lem:h_2estimate}
There exist $C,\epsilon>0$ so that if $\|h\|_{C^{2,\alpha}}<\epsilon$ and $h_2$ is as above, then
	\begin{equation}
		\|h_2\|_{C^{2,\alpha}}\leq C(\|h\|^3_{C^{0,\alpha}}+\|\Phi(g+h)\|_{C^{0,\alpha}}).
	\end{equation}
\end{lemma}
\begin{proof}
Using Lemma \ref{lem:elliptic-holder}, and (\ref{eq:holder-2}), we have 
\begin{align*}
\|h_2\|_{C^{2,\alpha}} &\leq&& C\|Lh_2\|_{C^{0,\alpha}} \leq C'( \|\Phi(g+h)\|_{C^{0,\alpha}}+\|h\|^3_{C^{2,\alpha}}+\|h_1\|_{C^{2,\alpha}}^2 + \|h_1\|_{C^{2,\alpha}} \|v\|_{C^{2,\alpha}} )\\
&\leq && C''( \|\Phi(g+h)\|_{C^{0,\alpha}}+\|\Phi(g+h)\|_{C^{0,\alpha}}^3 +\|h\|^3_{C^{0,\alpha}})\\&&& + C''(\|\Phi(g+h)\|_{C^{0,\alpha}}^2 +\|h\|^4_{C^{0,\alpha}} +\|\Phi(g+h)\|_{C^{0,\alpha}}\|h\|_{C^{0,\alpha}}+\|h\|^3_{C^{0,\alpha}})&& .
\end{align*}
In the last line we have used Lemma \ref{lem:h_1estimate} for the $h,h_1$ terms and Lemma \ref{lem:k1bound} to estimate $\|v\|_{C^{2,\alpha}} \leq C \|h\|_{C^{0,\alpha}}$. Absorbing higher order terms then yields the result. 
\end{proof}

We also record estimates on $u,k_2$:

\begin{lemma}
\label{lem:k_2estimate}
There exist $C,\epsilon>0$ so that if $\|h\|_{C^{2,\alpha}}<\epsilon$ and $u,k_2$ are as above, then
	\begin{equation}
		\|u\|_{C^{2,\alpha}}+ \|k_2\|_{C^{2,\alpha}}\leq C\|h\|_{C^{0,\alpha}}^2,\end{equation}
\end{lemma}

\begin{proof}
	By the elliptic estimate Lemma \ref{lem:elliptic-holder}, we have 
	
		\[
		\|u\|_{C^{2,\alpha}}+\|k_2\|_{C^{2,\alpha}} \leq C\| L(ug+k_2)\|_{C^{2,\alpha}} \leq C' \| \mathcal{D}^2\Phi( vg,vg)\|_{C^{0,\alpha}}.\]
		The uniform bounds on $\mathcal{D}^2\Phi$ and Lemma \ref{lem:k1bound} then give \[\| \mathcal{D}^2\Phi( vg,vg)\|_{C^{0,\alpha}} \leq C\|v\|_{C^{2,\alpha}}^2 \leq C' \|h\|^2_{C^{0,\alpha}}.\] 
\end{proof}

\subsubsection{3rd order analysis}
By Lemma \ref{lem:taylor-exp-holder} with $k=3$, we have
\[
\|\Phi(g+h)-\Phi(g)-\mathcal{D}\Phi(h)-\frac{1}{2}\mathcal{D}^2\Phi(h,h)-\frac{1}{6}\mathcal{D}^3\Phi(h,h,h)\|_{C^{0,\alpha}}\leq C\|h\|^4_{C^{2,\alpha}}.
\]
Using our decomposition $h=vg+h_1$, $h_1 = ug+k_2+h_2$ this becomes
\[
	\|\Phi(g+h)-Lh_2-\frac{1}{2}\mathcal{D}^2\Phi(h_1,h_1)
	-\mathcal{D}^2\Phi(h_1,vg)-\frac{1}{6}\mathcal{D}^3\Phi(h,h,h)\|_{C^{0,\alpha}}\leq C\|h\|^4_{C^{2,\alpha}}.
\]
By the uniform estimates on $\mathcal{D}^k\Phi$, after expanding further we get 
\begin{equation}
\label{eq:holder-3}
\begin{split}
\| \Phi(g+h) - Lh_2&-\mathcal{D}^2\Phi(ug+k_2,vg)-\frac{1}{6}\mathcal{D}^3\Phi(vg,vg,vg)\|_{C^{0,\alpha}}\\&\leq C(\|h\|^4_{C^{2,\alpha}} + \|h_1\|_{C^{2,\alpha}}^2 + \|h_2\|_{C^{2,\alpha}} \|v\|_{C^{2,\alpha}}) \\&\quad + C(\|v\|_{C^{2,\alpha}}^2 \| h_1\|_{C^{2,\alpha}} + \|v\|_{C^{2,\alpha}} \|h_1\|_{C^{2,\alpha}}^2 + \|h_1\|_{C^{2,\alpha}}^3). 
\end{split}
\end{equation}

\begin{proof}[Proof of Theorem \ref{thm:quantitative rigidity}]
Decompose $h = vg + ug+k_2+h_2$ as above. By the formal obstruction Corollary \ref{cor:formal-obstruction} and equivalence of norms on the finite-dimensional space $\mathcal{K}_1 \subset\mathcal{K}$, we have \begin{equation}\begin{split} \|v\|_{C^{2,\alpha}}^3 &\leq C \| \pi_{\mathcal{K}}( \mathcal{D}^2\Phi(ug+h,vg)+ \frac{1}{6}\mathcal{D}^3\Phi(vg,vg,vg))\|_{C^{2,\alpha}} \\&\leq \| Lh_2 + \mathcal{D}^2\Phi(ug+k_2, vg) + \frac{1}{6}\mathcal{D}^3\Phi(vg,vg,vg) \|_{C^{0,\alpha}}. \end{split}\end{equation} In the second line we have used Lemma \ref{lem:k1bound} and that $\pi_\mathcal{K}(Lh_2)=0$ (by self-adjointness). Then by the triangle inequality and (\ref{eq:holder-3}), we have 
\begin{equation}
\begin{split}
\|v\|_{C^{2,\alpha}}^3 \leq &C(\|\Phi(g+h)\|_{C^{0,\alpha}} + \|h\|^4_{C^{2,\alpha}} + \|h_1\|_{C^{2,\alpha}}^2 + \|h_2\|_{C^{2,\alpha}} \|v\|_{C^{2,\alpha}}) \\&\quad + C(\|v\|_{C^{2,\alpha}}^2 \| h_1\|_{C^{2,\alpha}} + \|v\|_{C^{2,\alpha}} \|h_1\|_{C^{2,\alpha}}^2 + \|h_1\|_{C^{2,\alpha}}^3). 
\end{split}
\end{equation}

Using Lemmas \ref{lem:k1bound}, \ref{lem:h_1estimate} and \ref{lem:h_2estimate} this gives that 

\begin{equation*}
\begin{split}
\frac{1}{C} \|v\|_{C^{2,\alpha}}^3 \leq& \|\Phi(g+h)\|_{C^{0,\alpha}} + \|\Phi(g+h)\|^4_{C^{0,\alpha}} + \|h\|^4_{C^{0,\alpha}} + \|\Phi(g+h)\|^2_{C^{0,\alpha}} + \|h\|^4_{C^{0,\alpha}} \\& + \|h\|_{C^{0,\alpha}}(\|\Phi(g+h)\|_{C^{0,\alpha}} + \|h\|^3_{C^{0,\alpha}}) + \|h\|_{C^{0,\alpha}}^2 (\|\Phi(g+h)\|_{C^{0,\alpha}} + \|h\|^2_{C^{0,\alpha}} ) \\& + \|h\|_{C^{0,\alpha}}(\|\Phi(g+h)\|^2_{C^{0,\alpha}} + \|h\|^4_{C^{0,\alpha}} ) + \|\Phi(g+h)\|^3_{C^{0,\alpha}} + \|h\|^6_{C^{0,\alpha}}.
\end{split}
\end{equation*}

Absorbing higher order terms yields 
\begin{equation}
\|v\|_{C^{2,\alpha}}^3 \leq \|\Phi(g+h)\|_{C^{0,\alpha}} +  \|h\|^4_{C^{0,\alpha}} .
\end{equation}

Finally, to estimate $h = vg + ug+k_2+h_2$, we combine this with Lemmas \ref{lem:h_2estimate} and \ref{lem:k_2estimate} to find
\begin{equation}
\begin{split}
\|h\|_{C^{2,\alpha}} &\leq   \|v\|_{C^{2,\alpha}} + \|ug+k_2\|_{C^{2,\alpha}} + \|h_2\|_{C^{2,\alpha}} 
 \\&\leq C( \|\Phi(g+h)\|_{C^{0,\alpha}}^\frac{1}{3} +  \|h\|^\frac{4}{3}_{C^{0,\alpha}} + \|h\|_{C^{0,\alpha}}^2 + \|h\|_{C^{0,\alpha}}^3 + \|\Phi(g+h)\|_{C^{0,\alpha}}).
\end{split}
\end{equation}

Absorbing higher order terms again gives the desired estimate 
\begin{equation}
\|h\|_{C^{2,\alpha}} \leq C \|\Phi(g+h)\|_{C^{0,\alpha}}^\frac{1}{3}. 
\end{equation} 

\end{proof}

The proof of the quantitative rigidity theorem follows by combining Theorem \ref{thm:quantitative rigidity} and the slicing Lemma \ref{lem:slice}.

\begin{proof}[Proof of Theorem \ref{thm:main3}]
Let $\delta<\epsilon_1$, where $\epsilon_1$ is as in Theorem \ref{thm:quantitative rigidity}. Then by Lemma \ref{lem:slice}, there exist $c\in (1-\delta,1+\delta)$, $\psi \in \Diff(M)$ and $h\in\mathring{\mathcal{S}}^2_g \oplus (\Lap+1)\mathring{C}^\infty_g\cdot g$  so that $c\psi^* \tilde{g} = g+h$ and $\|h\|_{C^{2,\alpha}} <\epsilon_1$. By the symmetries of $\Phi$, we have $\|\Phi(c\psi^*\tilde{g})\|_{C^{0,\alpha}} \leq (1+\delta) \|\Phi(\tilde{g})\|_{C^{0,\alpha}}$. The result then follows from Theorem \ref{thm:quantitative rigidity}. 

\end{proof}

\appendix

\section{Variation of basic geometric quantities}
\label{sec:basic-var}

In this section we compute various variation formulae for the geometric quantities $\Ric_g, \Hess^g$ and their traces $R_g, \Lap^g$. 

\subsection{First variation}

Here we list the well-known first variations of the basic geometric quantities. 

\label{sec:basic-1st}

\begin{lemma}
\label{lem:basic-1st}
Consider a manifold $M$ with a 1-parameter family of metrics $g(s)$. Then at $s=0$, if $e_i$ is an orthonormal frame, we have 

\begin{equation}\Ric'_{ij} = -R_{ijkl} g'_{kl} + \frac{1}{2}(\nabla_i \nabla_k g'_{jk} + \nabla_j \nabla_k g'_{ik} + R_{ik} g'_{jk} + R_{jk}g'_{ik} - \Lap g'_{ij} - \nabla_i \nabla_j (\tr_g g').\end{equation}

\begin{equation} R' = -2 g'_{ij} R_{ij} + \div_g \div_g g'  + \langle \Ric, g'\rangle - \Lap (\tr_g g').\end{equation}

\begin{equation}\label{eq:1st-var-Hess}({\Hess'} \phi)_{ij} = -\frac{1}{2}(\nabla_i g'_{jk} + \nabla_j g'_{ik} -\nabla_k g'_{ij})\nabla_k \phi.\end{equation}

\begin{equation}\Lap' \phi = -g'_{ij}(\Hess \phi)_{ij} - \langle \div_g g', d\phi\rangle + \frac{1}{2} \langle \nabla \tr_g g', \nabla\phi\rangle.\end{equation}

\end{lemma}

\subsection{Conformal direction}

Here we consider higher order variations under conformal change. 

\label{sec:basic-conf}

\begin{lemma}
\label{lem:basic-conf}
Let $M$ be a manifold with the metric family $(1+sv)g$, where $v$ (and $\phi$) is a smooth function on $M$. Then we have the variation formulae

\begin{equation*}
\begin{split}
 \Ric_s &=   -\frac{n-2}{2} \Hess v - \frac{1}{2}(\Lap v)g,\\
 \Ric_{ss} &= (n-2)v\Hess v  +3 \frac{n-2}{2} dv\otimes dv + (v\Lap v  -\frac{n-4}{2} |dv|^2 )g,\\
 \Ric_{sss} &=-3(n-2)v^2 \Hess v -  9(n-2) vdv\otimes dv    - ( 3v^2 \Lap v +3(n-4) v|dv|^2)g,
\end{split}
\end{equation*}

\begin{equation*}
\begin{split}
R_s &=  -vR - (n-1)\Lap v ,\\
R_{ss} &= 2v^2R + 4(n-1) v\Lap v - \frac{(n-1)(n-6)}{2}|dv|^2 ,\\
R_{sss} &=  -6v^3 R - 18(n-1) v^2 \Lap v + \frac{9}{2}(n-1)(n-6)v|dv|^2,
\end{split}
\end{equation*}

\begin{equation*}
\begin{split}
\Hess_s \phi &=  \frac{1}{2}(-dv\otimes d\phi  - d\phi\otimes dv + \langle d\phi,dv\rangle g) ,\\
\Hess_{ss} \phi &= -v(-dv\otimes d\phi  - d\phi\otimes dv + \langle d\phi,dv\rangle g) ,\\
\Hess_{sss} \phi &=  3v^2(-dv\otimes d\phi  - d\phi\otimes dv + \langle d\phi,dv\rangle g),
\end{split}
\end{equation*}

\begin{equation*}
\begin{split}
\Lap_s \phi &= -v\Lap \phi + \frac{n-2}{2} \langle dv,d\phi\rangle ,\\
\Lap_{ss} \phi &= 2v^2\Lap \phi -2(n-2)\langle d\phi,dv\rangle  ,\\
\Lap_{sss} \phi &= -6v^3\Lap \phi + 9(n-2)v^2\langle d\phi,dv\rangle.
\end{split}
\end{equation*}

\end{lemma}
\begin{proof}

The relevant quantities are given explicitly under conformal variation $\wt{g}= e^{2\psi} g$ by

\[\wt{\Ric} = \Ric - (n-2)(\Hess \psi - d\psi\otimes d\psi) - (\Lap \psi + (n-2) |d\psi|^2)g,\] 
\[\wt{R} = e^{-2\psi}( R - 2(n-1)\Lap \psi -(n-1)(n-2) |d\psi|^2),\]
\[\wt{\Hess}f = \Hess f - d\psi \otimes df - df\otimes d\psi + \langle \nabla f,\nabla \psi\rangle g,\] 
\[ \wt{\Lap} f = e^{-2\psi}(\Lap f + (n-2)\langle \nabla \psi,\nabla f\rangle).\] 

The above variation formulae follow by setting $e^{2\psi} = (1+sv)$ and expanding in $s$ to third order. Note that $\phi = \frac{1}{2} \log(1+sv) = \frac{s}{2}v - \frac{s^2}{4}v^2 + \frac{s^3}{6}v^3 - \cdots$. 

\end{proof}

\subsubsection{Second variation cross term}
\label{sec:conf-var-cross}

\begin{lemma}
Let $M$ be a manifold with the 2-parameter family of metrics $(1+sv+tu)g$, where $u,v$ are smooth functions on $M$. Then we have
\[ \Ric_{st} = \frac{3(n-2)}{4}(du\otimes dv + dv\otimes du) + \frac{n-2}{2}(u\Hess v + v\Hess u) + \frac{1}{2}(u\Lap v+ v\Lap u)g - \frac{n-4}{2}\langle du,dv\rangle g,\]
\[\Hess_{st}\phi = \frac{1}{2}(d(uv)\otimes d\phi + d\phi \otimes d(uv) -\frac{1}{2}\langle d\phi,d(uv)\rangle g,\]
\[R_{st} = 2uvR + 2(n-1)(u\Lap v+v\Lap u) - \frac{(n-1)(n-6)}{2}\langle du,dv\rangle.\] 
\end{lemma}
\begin{proof}
Polarise the second variation formulae in the previous lemma. 
\end{proof}

\subsection{Second variation in kernel-TT direction}

\label{sec:basic-ker-TT}

Here we compute the basic variations needed in Section \ref{sec:2var-ker-tt}, by considering the 2-parameter variation $(1+sv)(g+th)$. This will allow us to use the variation formulae under conformal change, although now $g_{st}=vh$ so for any quantity $Q$ we will have $Q_{st} = \mathcal{D}^2Q(vg,h) + \mathcal{D}Q(vh)$. 

\begin{lemma}
Let $(M,g)$ be a closed 1-Einstein manifold, and $v\in \ker(\Lap+2)$ and $h\in \mathring{\mathcal{S}}^2_g$. Then for any smooth function $u$ on $M$, we have
\[ \langle \mathcal{D}^2\Ric(vg,h), ug\rangle_{L^2} = \frac{n-2}{2} \int u\langle h,\Hess v\rangle w\d V_g,\]

\[\mathcal{D}^2 R(vg,h) = (n-2)\langle h,\Hess v\rangle.\] 
\end{lemma}
\begin{proof}
First, by the conformal variation Lemma \ref{lem:basic-conf}, at $s=0$ we have \[\Ric_s = -\frac{n-2}{2} \Hess^{g+th}v - \frac{1}{2} (\Lap^{g+th}v)(g+th).\] Since $\tr h =0,\div h=0$, by the general first variation Lemma \ref{lem:basic-1st} we have $\Hess_t=0$ and hence $\Lap_t \phi = -\langle h, \Hess \phi\rangle$. So differentiating in $t$ now, at $(s,t)=(0,0)$ we have \[\Ric_{st} = \frac{1}{2} \langle h,\Hess v\rangle g - \frac{1}{2}(\Lap v)h.\]

This implies \[\langle \Ric_{st}, ug\rangle_{L^2} = \int u\left( \frac{n}{2}\langle h,\Hess v\rangle - \frac{1}{2}(\Lap v)\tr h\right)w \d V_g = \frac{n}{2} \int u\langle h,\Hess v\rangle w\d V_g.\] By Lemma \ref{lem:basic-1st} again we have \[\begin{split}\langle \mathcal{D}\Ric(vh), ug\rangle_{L^2} &= -\int vu\langle \Ric,h\rangle w\d V_g+  \int u ( \div\div(vh) + v\langle \Ric,h\rangle - \Lap(\tr vh)) w\d V_g \\&= \frac{1}{2} \int u\langle h,\Hess v\rangle w \d V_g,\end{split}\] 
where we used that the initial metric satisfies $\Ric=g$. 
Therefore
\[ \langle \mathcal{D}^2\Ric(vg,h), ug\rangle_{L^2} = \langle \Ric_{st}-\mathcal{D}\Ric(vh), ug\rangle_{L^2} = \frac{n-2}{2} \int u\langle h,\Hess v\rangle w\d V_g.\] 

Now for any 2-tensor $T$, in an orthonormal frame at $(s,t)=(0,0)$ we have \[(\tr T)_{st} = \tr_g(T_{st})-\langle g_t, T_s\rangle - \langle g_s, T_t\rangle - \langle g_{st},T\rangle + 2(g_s)_{ij} (g_t)_{ik} T_{kj}.\] 

In particular this gives \[R_{st} = \frac{n}{2} \langle h,\Hess v\rangle - \frac{1}{2} (\Lap v)\tr h - \langle h, -\frac{\Lap v}{2} g - \frac{n-2}{2}\Hess v\rangle - v \tr \Ric_t - \langle vh,\Ric\rangle + 2v \langle h,\Ric\rangle.\] Again using that $\Ric=g$ and $\tr h=0,\div g=0$, we get $R_{st} = \frac{2n-2}{2} \langle h,\Hess v\rangle$. By Lemma \ref{lem:basic-1st} we have $\mathcal{D}R(vh) = \div\div(vh) = \langle h,\Hess v\rangle$, therefore 

\[\mathcal{D}^2 R(vg,h) = R_{st} - \mathcal{D}R(vh) = (n-2)\langle h,\Hess v\rangle.\] 


\end{proof}

\section{Uniform bounds for shrinker quantities}
\label{sec:shrinker-quantities-bound}

Here we prove Lemma \ref{lem:shrinker-quantities-bound}, which is reproduced below. Recall that $T_1 * T_2$ denotes an unspecified contraction of the tensors $T_1,T_2$, and $\|T_1 * T_2 \|_{C^{k,\alpha}} \leq \|T_1\|_{C^{k,\alpha}} \|T_2\|_{C^{k,\alpha}}$. 

\begin{lemma}[= Lemma \ref{lem:shrinker-quantities-bound}]
Fix a reference soliton metric $\Phi(g_0)=0$. There exists $C_k, \epsilon>0$ such that if $\|g-g_0\|_{C^{2,\alpha}} <\epsilon$, then for any $k$ the variations at $g$ satisfy
\begin{equation}  \|\mathcal{D}^k\nu_g\| + \|\mathcal{D}^k\tau_g\|+\|\mathcal{D}^k f_g\|_{C^{2,\alpha}}+\|\mathcal{D}^k \Phi_g\|_{C^{0,\alpha}} \leq C_k ,\end{equation}
\end{lemma}

\begin{proof}
We proceed by induction on $k$. For brevity we abbreviate $\mathbf{h}= (h_1,\cdots,h_k)$ and $\|\mathbf{h}\|^k = \prod_{i=1}^k \|h_i\|$ . The main idea is that we can estimate derivatives of $\tau,f$ by differentiating their defining equations in the $\mathbf{h}$ direction and focussing on the highest order terms; all other terms will be estimated by the inductive hypothesis. 

The case $k=0$ is clear as each quantity depends smoothly on $g$. So suppose that \[\|\mathcal{D}^j\nu_g\| + \|\mathcal{D}^j\tau_g\|+\|\mathcal{D}^j f_g\|_{C^{2,\alpha}}+\|\mathcal{D}^j \Phi_g\|_{C^{0,\alpha}} \leq C_j\] for all $j<k$. For a tensor quantity $T$ we denote an unspecified linear combination of contractions of copies of $\mathcal{D}^\cdot T$ with total order $j$ by
 \[(*\mathcal{D})^jT= \sum_{j_1+\cdots + j_\alpha=j} \mathcal{D}^{j_1}T * \cdots * \mathcal{D}^{j_l}T.\] In particular, if $j_1+j_2+j_3= j<k$, $\mathbf{h}'\subset\mathbf{h}$ then by the inductive hypothesis
 \[  \left\| \left((*\mathcal{D})^{j_1}\tau * (*\mathcal{D})^{j_2} f * (*\mathcal{D})^{j_3}\Phi\right)( \mathbf{h}') \right\|_{C^{0,\alpha}} \leq C_j \|\mathbf{h}'\|^j_{C^{2,\alpha}}.\] 

Recall that $\mathcal{D}\nu_g (h) = -\int \langle \Phi(g),h\rangle w^g \d V_g$, where $w=(4\pi\tau)^{-n/2} e^{-f}$. Differentiating this further, we find that \[ \mathcal{D}^k\nu_g(\mathbf{h}) = - \int  w\d V_g \left( h_1 * \left( \sum_{k_1+k_2+k_3=k-1} (*\mathcal{D})^{k_1}\tau * (*\mathcal{D})^{k_2} f * (*\mathcal{D})^{k_3}\Phi\right)(\mathbf{h}') \right) ,\] where $\mathbf{h}' = (h_2,\cdots,h_k)$. It follows immediately that $| \mathcal{D}^k\nu_g(\mathbf{h})| \leq C \|\mathbf{h}\|_{C^{2,\alpha}}^k$. 

By differentiating the normalisation $\int w^g \d V_g=1$, we have \[-\frac{n}{2\tau} \mathcal{D}^k\tau(\mathbf{h}) - \int \mathcal{D}^k f(\mathbf{h}) w\d V_g = \int w \d V_g\left(\sum_{\substack{k_1+k_2=k \\ k_1,k_2<k}} (*\mathcal{D})^{k_1}\tau * (*\mathcal{D})^{k_2}f\right)(\mathbf{h}).\] Similarly by differentiating $\int f^gw^g \d V_g=\frac{n}{2}+\nu(g)$, we have \[\begin{split} -\frac{n}{2\tau} \mathcal{D}^k\tau (\mathbf{h}) +& \left(1-\nu -\frac{n}{2}\right) \int \mathcal{D}^k f(\mathbf{h}) w \d V_g \\&= \mathcal{D}^k\nu(\mathbf{h})+ \int w \d V_g\left( \sum_{\substack{k_1+k_2=k \\ k_1,k_2<k}} (*\mathcal{D})^{k_1}\tau *  (*\mathcal{D})^{k_2} f \right)(\mathbf{h}).\end{split}\] 
Eliminating $\int \mathcal{D}^k f(\mathbf{h}) w\d V_g$, we then have \[ \frac{\frac{n}{2}+\nu-2}{\frac{2\tau}{n}} \mathcal{D}^k\tau(\mathbf{h})  = \mathcal{D}^k\nu(\mathbf{h})+ \int w \d V_g\left( \sum_{\substack{k_1+k_2=k \\ k_1,k_2<k}} (*\mathcal{D})^{k_1}\tau *  (*\mathcal{D})^{k_2} f \right)(\mathbf{h}) .\] It follows that  $|\mathcal{D}^k\tau_g(\mathbf{h})| \leq  C \|\mathbf{h}\|_{C^{2,\alpha}}^k$. 

Recall the defining equation (\ref{eq:EL-nu}) for $f$, which states \[\tau^g(-2\Lap^g f^g + |df^g|_g^2 -R_g) - f^g+n +\nu(g)=0.\] All the basic geometric quantities involve at most the 2-jet of $g$. Let $\Lap_f:= \Lap - \langle df, d(\cdot)\rangle$ denote the drift Laplacian. 
By differentiating the defining equation, we will have \[\begin{split} (\Lap_f+\frac{1}{2\tau}) \left(\mathcal{D}^k f(\mathbf{h})\right)  =& \frac{1}{2\tau} \mathcal{D}^k\nu(\mathbf{h}) +\frac{1}{2\tau} \mathcal{D}^k\tau(\mathbf{h}) (-2\Lap f + |df|^2-R) \\&+ \sum_{\substack{\mathbf{h}' \cup \mathbf{h}'' = \mathbf{h} \\ \# \mathbf{h}' = k_1 \\ \# \mathbf{h}''=k_2} }
(*\mathcal{J}_2)(\mathbf{h}') *\left( \sum_{\substack{j_1+j_2= k_2\\ j_2<k}} (*\mathcal{D})^{j_1}\tau * (*\mathcal{D})^{j_2} \mathcal{J}_2(f)\right)(\mathbf{h}'').\end{split}\] Here $\mathcal{J}_2(f)$, $\mathcal{J}_2(h)$ denote the 2-jets of $f$ and $h$ (with respect to the reference metric $g_0$ on $M$), $(*\mathcal{J}_2)(h_1,\cdots, h_j) = \mathcal{J}_2(h_1) \cdots * \mathcal{J}_2(h_j)$ and the sum is over all partitions of $\mathbf{h}$. 

By the previous estimates, we then have 
\[\begin{split}\|(\Lap_f+\frac{1}{2\tau}) &\left(\mathcal{D}^k f( \mathbf{h})\right)\|_{C^{0,\alpha}} \\&\leq C(|\mathcal{D}^k\nu(\mathbf{h})|  + | \mathcal{D}^k\tau(\mathbf{h})|) \\&\quad+C\sum_{\substack{\mathbf{h}' \cup \mathbf{h}''= \mathbf{h} \\ \# \mathbf{h}' = k_1 \\ \# \mathbf{h}''=k_2<k}} \|\mathbf{h}' \|^{k_1}_{C^{2,\alpha}} \left\| \left(\sum_{j_1+j_2= k_2} (*\mathcal{D})^{j_1}\tau * (*\mathcal{D})^{j_2} \mathcal{J}_2(f)\right)(\mathbf{h}'')\right\|_{C^{0,\alpha}} \\&\leq C' \|\mathbf{h}\|^k_{C^{2,\alpha}} .\end{split}\]

For a shrinking soliton, it is known that $\frac{1}{2\tau^{g_0}} \notin \spec \Lap^{g_0}_f$ (see for instance \cite{CZ}). Therefore, by continuity, for $\epsilon$ small enough we will have $\frac{1}{2\tau^g} \notin \spec\Lap^g_f$. In particular, $(\Lap_f+\frac{1}{2\tau})$ is coercive (uniformly in $g$), so by elliptic Schauder theory 
\[
\| \mathcal{D}^k f_g (\mathbf{h}) \|_{C^{2,\alpha}} \leq C\|(\Lap^g_f+\frac{1}{2\tau^g}) \left(\mathcal{D}^k f_g( \mathbf{h})\right)\|_{C^{0,\alpha}}  \leq C' \|\mathbf{h}\|^k_{C^{2,\alpha}}.
\]

Finally, note that again all the basic geometric quantities in the definition \[\Phi(g) = \tau^g(\Ric(g) +\Hess^g f^g) - \frac{g}{2}\] depend on at most the 2-jet of $g$. Then \[ \mathcal{D}^k \Phi (\mathbf{h}) =  \sum_{\substack{\mathbf{h}' \cup \mathbf{h}''= \mathbf{h} \\ \# \mathbf{h}' = k_1 \\ \# \mathbf{h}''=k_2}} (*\mathcal{J}_2)(\mathbf{h}') * \left(\sum_{j_1+j_2=k_2} (*\mathcal{D})^{j_1}\tau * (*\mathcal{D})^{j_2}f\right)(\mathbf{h}'')+ h_1* \cdots * h_k.\] By the previous estimates we conclude that indeed $\|\mathcal{D}^k \Phi_g (\mathbf{h})\|_{C^{0,\alpha}} \leq C \|\mathbf{h}\|^k_{C^{2,\alpha}}$, which completes the induction. 
\end{proof}

\bibliographystyle{amsalpha}
\bibliography{ricci-soliton}

\end{document}